\documentclass{amsart}
\usepackage{url}
\usepackage{hyperref,amssymb}
\hypersetup{colorlinks=true,linkcolor=blue,urlcolor=cyan,citecolor=magenta} 
\usepackage[hyperpageref]{backref}
\usepackage{frcursive}
\usepackage{calrsfs}
\usepackage{mathrsfs}
\usepackage{multicol}
\newtheorem{theorem}{Theorem}[section]
\newtheorem{lemma}[theorem]{Lemma}
\newtheorem{remark}[theorem]{Remark}
\newtheorem{example}[theorem]{Example}
\newtheorem{proposition}[theorem]{Proposition}
\newtheorem{corollary}[theorem]{Corollary}
\newtheorem{conjecture}[theorem]{Conjecture}
\theoremstyle{definition}
\newtheorem{definition}[theorem]{Definition}
\numberwithin{equation}{section}
\newcommand{\Q}{\mathbb{Q}}
\newcommand{\Z}{\mathbb{Z}}
\newcommand{\F}{\mathbb{F}}

\newcommand{\ds}{\displaystyle}
\newcommand{\ov}{\overline}
\newcommand{\wt}{\widetilde}
 
\newcommand{\sst}\scriptstyle
\newcommand{\ft}{\footnotesize}
\newcommand{\ns}\normalsize

\newcommand{\cl}{c\hskip-1pt{\ell}}
\newcommand{\order}{\raise0.8pt \hbox{${\sst \#}$}}
\newcommand{\oorder}{\raise0.8pt \hbox{\tiny${\sst \#}$}}
\newcommand{\lien}{\mathrel{\mkern-4mu}}
\newcommand{\too}{\relbar\lien\rightarrow}

\newcommand{\plus}{\ds\mathop{\raise 2.0pt \hbox{$\bigoplus$}}\limits}
\newcommand{\prd}{\ds\mathop{\raise 2.0pt \hbox{$\prod$}}\limits}
\newcommand{\sm}{\ds\mathop{\raise 2.0pt \hbox{$\sum$}}\limits}
\newcommand{\ffrac}[2]{\hbox{\ft $\displaystyle\frac{#1}{#2}$}}
\newcommand{\limproj}{\mathop{\lim_{\longleftarrow}}}

\newcommand{\Gal}{{\rm Gal}}
\newcommand{\Norm}{{\rm N}}

\newcommand{\J}{{\rm J}}

\newcommand{\pr}{{\rm pr}}

\newcommand{\ab}{{\rm ab}}

\newcommand{\ta}{{\rm ta}}

\newcommand{\rk}{{\rm rk}}
\newcommand{\tor}{{\rm tor}}
\newcommand{\Hom}{{\rm H}}
\newcommand{\Ker}{{\rm Ker}}
\newcommand{\Cha}{{\rm III}}
\newcommand{\es}{{\emptyset}}


\begin{document}

\title[On the $\lambda$-stability of $p$-class groups]
{On the $\lambda$-stability of $p$-class groups \\ along cyclic $p$-towers of a number field}

\author{Georges Gras}
\address{Villa la Gardette, 4, Chemin Ch\^ateau Gagni\`ere,
F-38520 Le Bourg d'Oisans -- \url{http://orcid.org/0000-0002-1318-4414}} 
\email{g.mn.gras@wanadoo.fr}

\date{August 12, 2021}

\begin{abstract}
Let $k$ be a number field, $p \geq 2$ a prime and $S$ a set of tame or wild 
finite places of $k$. We call $K/k$ a totally $S$-ramified cyclic $p$-tower if 
$\Gal(K/k) \simeq \Z/p^N\Z$ and if $S \ne \es$ is totally ramified. 
Using analogues of Chevalley's formula (Gras, Proc. Math. Sci. {\bf 127}(1) (2017)),
we give an elementary proof of a stability theorem (Theorem \ref{mt}) 
for generalized $p$-class groups ${\mathcal X}_n$ of the layers $k_n \subseteq K$:
let $\lambda = \max(0, \oorder S-1-\rho)$ given in Definition \ref{Lambda}; then
$\order {\mathcal X}_n = \order {\mathcal X}_0 \cdot p^{\lambda \cdot n}$ 
for all $n \in [0,N]$, if and only if $\order {\mathcal X}_1 = 
\order {\mathcal X}_0 \cdot p^{\lambda}$. This improves the case $\lambda = 0$ of 
Fukuda (1994), Li--Ouyang--Xu--Zhang (2020), Mizusawa--Yamamoto (2020),
whose techniques are based on Iwasawa's theory or Galois theory of pro-$p$-groups. 
We deduce capitulation properties of ${\mathcal X}_0$ in the tower. Finally we apply our 
principles to the torsion groups ${\mathcal T}_n$ of abelian $p$-ramification theory. 
Numerical examples are given.
\end{abstract}

\maketitle

\tableofcontents

\section{Introduction} 
The behavior of $p$-class groups in a $\Z_p$-extension gives rise to
many theoretical and computational results; a main 
observation is the unpredictability of these groups in the 
first layers of the $\Z_p$-extension, then their regularization
from some not effective level, Iwasawa's theory giving the famous formula
of the orders by means of the invariants $\lambda, \mu, \nu$; 
nevertheless, this algebraic context is not sufficient to estimate these 
parameters (e.g., Greenberg's conjecture \cite{Gre}). 
Our purpose is to see, more generally, what is the behavior of the 
$p$-class groups (and some other arithmetic invariants) in cyclic 
$p$-extensions in which one allows tame ramification. 

\smallskip
Let $k$ be a number field, $p \geq 2$ a prime number and
$K$ a cyclic extension of degree $p^N$ of $k$, $N \geq 1$, 
and let $k_n$, $n \in [0, N]$, be the degree $p^n$ extension of $k$ in $K$. 
We speak of cyclic $p$-towers $K/k$ and put $G = \Gal(K/k)$. 
Let $S$ be the set of places of $k$, ramified in $K/k$. 
{\it We will assume $S \ne \es$ and $S$ totally ramified in $K/k$};
in general, $S$ contains tame places, except if $K$ is contained in 
a $\Z_p$-extension of $k$, in which case $S$ is a set of 
$p$-places of $k$. Due to the very nature of a tower, there will never be 
``complexification=ramification'' of infinite real places.

\smallskip
Let ${\mathfrak m} \ne 0$ be an integer ideal of $k$, 
of support $T$ {\it disjoint from $S$};
let ${\mathcal C}_{k,{\mathfrak m}}$ and ${\mathcal C}_{k_n,{\mathfrak m}}$, 
denoted simply ${\mathcal C}_{0,{\mathfrak m}}$ 
and ${\mathcal C}_{n,{\mathfrak m}}$, be the $p$-Sylow subgroups 
of the ray class groups modulo ${\mathfrak m}$, of $k$ and $k_n$, respectively
(for $n >0$, ${\mathfrak m}$ is seen as extended ideal in $k_n$).
The class of an ideal ${\mathfrak a}$ of $k_n$ is denoted 
$\cl_{n,{\mathfrak m}}({\mathfrak a})$.

\smallskip
 Let $\Norm_{K/k_n}$ be the 
arithmetic norm in $K/k_n$, defined on class groups from norms of ideals.
Since ${\mathfrak m}$ is prime to $S \ne \es$ totally ramified, the corresponding 
$p$-Hilbert's ray class fields $H_{n,{\mathfrak m}}$, of the $k_n$'s
are linearly disjoint from the relative $p$-towers $K/k_n$; since, by class 
field theory, $\Norm_{K/k_n}$ corresponds to the restriction of automorphisms 
$\Gal(H_{N,{\mathfrak m}}/K) \to \Gal(H_{n,{\mathfrak m}}/k_n)$,
we get $\Norm_{K/k_n} ({\mathcal C}_{N,{\mathfrak m}}) = 
{\mathcal C}_{n,{\mathfrak m}}$, for all $n \in [0, N]$.

\smallskip
Let ${\mathcal H}_N$ be a sub-$\Z_p[G]$-module
of ${\mathcal C}_{N,{\mathfrak m}}$;
we may represent {\it a minimal set} of $\Z_p[G]$-generators of ${\mathcal H}_N$
with prime ideals ${\mathfrak Q}_1, \ldots, {\mathfrak Q}_t \notin S \cup T$,
${\mathfrak Q}_i \mid {\mathfrak q}_i$ in $K/k$.
Let ${\mathcal I}_N$ be the $\Z[G]$-module generated by
these ideals; thus ${\mathcal I}_0 := \Norm_{K/k}({\mathcal I}_N)$
is of minimal $\Z$-rank 
$t$ since $\Norm_{K/k}({\mathcal I}_N) \otimes \Q =  
\langle {\mathfrak q}_1, \ldots, {\mathfrak q}_t \rangle_\Z \otimes \Q$. 
We set, for all $n \in [0, N]$:
\begin{equation}\label{DEF}
{\mathcal I}_n := \Norm_{K/k_n}({\mathcal I}_N), \  \ 
{\mathcal H}_n := \Norm_{K/k_n}({\mathcal H}_N), \  \ 
{\mathcal X}_{n,{\mathfrak m}} := {\mathcal C}_{n,{\mathfrak m}}/{\mathcal H}_n.
\end{equation}

This defines the family $\{{\mathcal X}_{n,{\mathfrak m}}\}_{n \in [0, N]}$
(``generalized $p$-class groups'') such that:
$$\hbox{$\Norm_{K/k_n} ({\mathcal X}_{N,{\mathfrak m}}) = 
{\mathcal X}_{n,{\mathfrak m}}$, for all $n \in [0, N]$.} $$
Any place ${\mathfrak Q}$ of $K$, such that 
$\cl_{N,{\mathfrak m}}({\mathfrak Q}) \in {\mathcal H}_N$, totally
splits in the subfield of $H_{N,{\mathfrak m}}$ fixed by the image 
of ${\mathcal H}_N$ in $\Gal(H_{N,{\mathfrak m}}/K)$; 
so this general definition allows to
enforce decomposition conditions in the ray class fields.

\smallskip
To simplify, we shall remove the indices ${\mathfrak m}$, except necessity. 

\smallskip
One speaks of ``$\lambda$-stability'' of these ${\mathcal X}_n$ 
along the cyclic $p$-tower $K/k$ when there exists $\lambda \geq 0$ such that
$\order {\mathcal X}_n = \order {\mathcal X}_0 \cdot p^{\lambda \cdot n}$, 
for all $n \in[0, N]$. Of course, this is not a workable definition, both theoretically
(classical context of $\Z_p$-extensions or less familiar case of $p$-towers with tame
ramification) and computationally (inability to use PARI/GP \cite{PARI} 
beyond some level $n = 3$ or $4$, and almost nothing for $p > 5$).
So we intend to get an accessible criterion likely to give information in all the tower.
 
\begin{definition}\label{Lambda}
Let $\Lambda := \{x \in k^\times,\, x \equiv 1\! \pmod {\mathfrak m},\ 
(x) \in \Norm_{K/k}({\mathcal I}_N)\}$. For instance, ${\mathcal I}_N=1$
yields $\{{\mathcal X}_n\} = \{{\mathcal C}_n\}$ and $\Lambda = E$, the group 
of units $\varepsilon \equiv 1 \!\! \pmod {\mathfrak m}$ of~$k$.
We have the exact sequence $1 \to \Lambda/E \to {\mathcal I}_0 \to {\mathcal H}_0 \to 1$.
Since $\Lambda/E$ is $\Z$-free (of $\Z$-rank $t$ because ${\mathcal H}_0$ is finite), one can 
write, with representatives $\alpha_j$ of $\Lambda/E$:
\begin{equation*}
\Lambda = \tor_{\Z}(E) \plus \langle \varepsilon_1, \ldots , \varepsilon_r;
\alpha_1, \ldots , \alpha_t \rangle_\Z =: \tor_\Z(E) \plus X, 
\ \, \varepsilon_i \in E,
\end{equation*}
$r := r_1+r_2-1$, where $(r_1, r_2)$ is the signature of $k$,
and where $X$ is $\Z$-free of $\Z$-rank $\rho := r+t$;  
we define the parameter $\lambda := \max\,(0, \order S -1 - \rho)$.

\smallskip
We shall see that the norm properties, in $K/k$, of $\tor_\Z(E) \otimes \Z_p$ 
are specific and may give some obstructions that are
elucidated by the following Lemma.
\end{definition}

\begin{lemma}(Albert's Theorem for $K/k$; for a proof, see \cite[Exercise I.6.2.3]{Gra3}).\label{albert} 
If $k$ contains  the group $\mu_{p^\epsilon}$ of $p^\epsilon$th roots of unity,
then $\mu_{p^\epsilon} \subset \Norm_{K/k}(K^\times)$ if and only if 
there exists a degree $p^\epsilon$ cyclic extension $L/K$ such that $L/k$ is cyclic.
\end{lemma}

\begin{remark}\label{torsion}
We shall restrict ourselves to the case
$\tor_\Z(E) \otimes \Z_p \subset \Norm_{K/k}(K^\times)$.
If necessary, it suffices to restrict the tower $K/k$ to the sub-tower $K'/k$ 
such that $[K : K']$ be the order of the group of $p$th 
roots of unity of $k$, or to notice that $L/K/k$ does exist (recall that no arithmetic
condition is required on $L/K$). 

\smallskip
Thus, the norm properties of $\Lambda$ are assigned to $X$.
\end{remark}
 
We put to simplify ${\mathcal C}_0 =: {\mathcal C}$, ${\mathcal H}_0 =: {\mathcal H}$, 
${\mathcal I}_0 =: {\mathcal I}$, ${\mathcal X}_0 =: {\mathcal X}$, 
and so on, ${\mathfrak m}$ being implied; let $p^e$ be the exponent of ${\mathcal X}$. 
Under the assumptions on $S$, $T$, 
we will prove the following result for the cyclic $p$-tower $K/k$ of degree $p^N$
(see Theorem \ref{mt} and Corollary \ref{capitulation} for more complete statements):

\smallskip\noindent
{\bf Main Result.}
{\it Assume that $\tor_\Z(E) \subset \Norm_{K/k}(K^\times)$ (Remark \ref{torsion}).
Let $\rho \geq 0$ be the $\Z$-rank of the  $\Z$-module $X$ and let 
$\lambda := \max\,(0, \order S -1 - \rho)$ (Definition \ref{Lambda}).
Then $\order {\mathcal X}_n = \order {\mathcal X} \cdot p^{\lambda \cdot n}$ 
for all $n \in [0, N]$), if and only if $\order {\mathcal X}_1 = 
\order {\mathcal X} \cdot p^{\lambda}$. If the criterion applies with $\lambda = 0$
(i.e., $\order {\mathcal X}_1 = \order {\mathcal X}$)
and if $e \leq N$, then ${\mathcal X}$ capitulates in $K$.}

\begin{remark}\label{lambda=0} 
({\bf a}) If $\lambda = \max\,(0, \order S -1 - \rho)$ 
fulfills the criterion of $\lambda$-stability from the base field $k$, it fulfills 
the same $\lambda$-stability criterion in any relative $p$-tower $K/k'$,
$k' := k_{n_0}$, $n_0 \in[0, N[$, since (with ${\mathcal X}'_n := {\mathcal X}_{n_0+n}$,
$n \geq 0$) we get $\order {\mathcal X}'_n = \order {\mathcal X}' \cdot
p^{\lambda \cdot n}$ since
$\order {\mathcal X}'_n = \order {\mathcal X}_{n_0+n} =
\order {\mathcal X} \cdot p^{\lambda\,(n_0+n)} = 
(\order {\mathcal X} \cdot  p^{\lambda\,n_0}) \cdot p^{\lambda\,n} =
\order {\mathcal X}'\cdot p^{\lambda\,n}$.

\smallskip\noindent
\quad ({\bf b})
If $\lambda = \max\,(0, \order S -1 - \rho)$ is not suitable for the criterion 
in $k_1/k$ (which shall be equivalent to $\order {\mathcal X}_1 > 
\order {\mathcal X} \cdot p^{\lambda}$ due to Chevalley's formula), 
we can consider a relative tower $K/k'$; then, in $k'$,
$\order S' = \order S$, $t'=t$, and $\lambda' := \max\,(0, \order S -1 - \rho')$ 
is a strictly decreasing function of $[k' : k]$ (indeed, from Definition \ref{Lambda},
we compute that $\rho' - \rho = r'-r = (r_1+r_2)\cdot ([k' : k]-1)$), so that two 
cases may arise:

\smallskip
\quad (i) For some $k' \subset K$, we have $\order {\mathcal X}'_1 = 
\order {\mathcal X}' \cdot p^{\lambda'}$ giving the $\lambda'$-stability in $K/k'$
with regular linear increasing orders from $k'$.

\smallskip
\quad (ii) Whatever $k' \subset K$, $\lambda'$ is not suitable 
(i.e., $\order {\mathcal X}'_1 > \order {\mathcal X}' \cdot p^{\lambda'}$, even if
$\lambda' = 0$, which occurs rapidly for $k'$ high enough in the tower), which 
means that $\order {\mathcal X}_n$ is strictly increasing from some level $n_0$,
which is illustrated by numerical examples and can define a 
non-linear increasing (in a $\Z_p$-extension $\wt k$ this means ``$\wt \mu \ne 0$'').
It is important to note that a $\wt \lambda$-stability may exist from some level, 
with $\wt \lambda > \lambda'$; this is the case in a $\Z_p$-extension $\wt k$ with
Iwasawa's invariants $\wt \lambda \ne 0$ and $\wt \mu=0$.

\smallskip
Numerical experiments are out of reach as soon as $n > 3$ or $4$, but a 
reasonable heuristic is a tendency to stabilization in totally real $p$-towers, 
by analogy with the study of Greenberg's conjecture carried out in \cite{Gra2}.
In a different framework, let $k_n$ be the $n$th layer of the cyclotomic 
$\Z_p$-extension $k_\infty$ of $k$ and let $C_n$ be the whole class group of $k_n$;
then, for $\ell \ne p$, $\order (C_n \otimes \Z_\ell)$ stabilizes in $k_\infty$ 
(a deep result in the abelian case \cite{Wa}, extended in \cite{FS} 
to Iwasawa context in the $\Z_p \times \Z_\ell$-extension of $k$).

\smallskip\noindent
\quad ({\bf c}) For $k$ fixed, $\lambda$ is unbounded
(which only depends on the tower via $S$) and the condition 
$\order {\mathcal C}_1 = \order {\mathcal C}$
of the literature ($\lambda = t = 0$, thus $\order S \leq r_1+r_2$) 
is empty as soon as $\lambda \geq 1$, whence the 
interest of the factor $p^\lambda$ to get examples whatever $S$.
For known results (all relative to $\lambda = 0$), one may cite 
Fukuda \cite{Fu} using Iwasawa's theory, 
Li--Ouyang--Xu--Zhang \cite[\S\,3]{LOXZ} working in a non-abelian 
Galois context, in Kummer towers, via the use of the fixed points formulas 
\cite{Gra0, Gra1}, then Mizusawa \cite{Miz} above $\Z_2$-extensions, 
and Mizusawa--Yamamoto \cite{MY} for generalizations, including 
ramification and splitting conditions, via the Galois theory of pro-$p$-groups.

\smallskip\noindent
\quad ({\bf d}) If Lemma \ref{albert} does not apply, some counterexamples 
can arise. For instance, let $k=\Q(\sqrt{-55})$, $p=2$, $K = k (\mu_{17})$ of degree 
$2^4$ over $k$ ($-1 \notin \Norm_{K/k}(K^\times)$, $\order S = 2$, $\lambda=1$); 
the successive  $2$-structures are ${\sf [4], [8], [16], [32], [32]}$, for which 
$\order {\mathcal C}_n = \order {\mathcal C} \cdot 2^n$ holds for $n \in [0, 3]$, 
since $-1$ is norm in $k_3/k$, but not for $n=4$.
\end{remark}

To conclude, we will show that the behavior, in $K/k$, of the finite torsion groups 
${\mathcal T}_n$ of $p$-ramification theory (as ``dual'' invariants of $p$-class groups) 
gives much strongly increasing groups with possibly non-zero $\wt \mu$-invariants
in $\Z_p$-extensions $\wt k/k$.

\section{Around Chevalley's ambiguous class number formula}
The well-known Chevalley's formula \cite[p.\,402]{Che} 
is the pivotal element for a great lot of ``fixed points formulas''. 
For the ordinary sense, it takes into account the complexification
of real infinite places, which does not occur for us as explained in 
the Introduction. 
Let $L/F$ be a cyclic $p$-extension of Galois group $G$, $S$-ramified, 
non-complexified if $p=2$, and let $e_v(L/F)$ be the ramification 
index of $v \in S$; we have
$\order {\mathcal C}_L^G = \ffrac{\order {\mathcal C}_F \cdot 
\prod_{v \in S} e_v(L/F)} {[L : F] \cdot 
(E_F : E_F \cap \Norm_{L/F}(L^\times))}$, where ${\mathcal C}_F$,
${\mathcal C}_L$ denote the $p$-class groups of $F$, $L$, respectively,
and $E_F$ the group of units of $F$. 

\smallskip
Various generalizations of this formula were given (Gras \cite[Th\'eor\`eme 2.7, 
Corollaire 2.8]{Gra0}, \cite[Theorem 3.6, Corollaries 3.7, 3.9]{Gra1}, Jaulent
\cite[Chap. II, \S\,3]{Jau0}); mention their recent idelic proof by
Li--Yu \cite[Theorem 2.1 $\&$ \S\,2.3, Examples]{LY}. We shall use the 
following one for the layers of a $S$-ramified cyclic $p$-tower $K/k$:

\begin{proposition}
Assume $S \ne \es$ totally ramified in $L/F$.
Let  ${\mathfrak m}$, of support $T$ disjoint from $S$, be a modulus of $F$
and let ${\mathcal C}_F$, ${\mathcal C}_L$, be the $p$-Sylow subgroups of the
ray class groups modulo ${\mathfrak m}$, of $F$ and $L$, respectively. 
For any $\Z[G]$-module ${\mathcal I}_L$, of prime-to-$T \cup S$ ideals of $L$, 
defining ${\mathcal H}_L := \cl_L({\mathcal I}_L)$, ${\mathcal H}_F :=
\Norm_{L/F}({\mathcal H}_L) = \cl_F(\Norm_{L/F}({\mathcal I}_L))$, 
we have, with $\Lambda_{L/F} := \{x \in F^\times,\, x \equiv 1\!\! \pmod {\mathfrak m},\, 
(x) \in \Norm_{L/F}({\mathcal I}_L)\}$:
\begin{equation}\label{totram}
\order ({\mathcal C}_L /{\mathcal H}_L)^G
= \order( {\mathcal C}_F /{\mathcal H}_F) \times \frac{(\order G)^{\oorder S - 1}}
{(\Lambda_{L/F} : \Lambda_{L/F} \cap \Norm_{L/F}(L^\times))} .
\end{equation}
\end{proposition}

\begin{lemma} \label{inclusion} 
Let $K/k$ be a cyclic $p$-tower of degree $p^N$, of Galois group 
$G =: \langle \sigma \rangle$. Let ${\mathcal H}_N =: \cl_N ({\mathcal I}_N)
\subseteq {\mathcal C}_N$ and $\Lambda := 
\{x \in k^\times,\, x \equiv 1\! \pmod {\mathfrak m},\, (x) \in \Norm_{K/k}({\mathcal I}_N)\}$
(see Definition \ref{Lambda}):

\smallskip
(i) $\Lambda = \{x \in k^\times,\, x \equiv 1\! \pmod {\mathfrak m},\, 
(x) \in \Norm_{k_n/k}({\mathcal I}_n)\}$ (i.e., $\Lambda_{K/k} = \Lambda_{k_n/k}$), 

\smallskip
(ii) $\Lambda \subseteq \Lambda_1 := \{x_1 \in k_1^\times,\, 
x_1 \equiv 1\! \pmod {\mathfrak m},\, (x_1) \in \Norm_{K/k_1}({\mathcal I}_N)\}$.
\end{lemma}

\begin{proof}
Point (i) comes from ${\mathcal I}_n = \Norm_{K/k_n} ({\mathcal I}_N)$.
Let $x \in \Lambda$ seen in $k_1$; then $(x) = \Norm_{K/k}({\mathfrak A}) = 
\Norm_{K/k_1}({\mathfrak A}^{\Theta})$, where $\Theta=1+\sigma+ \cdots + \sigma^{p-1}$,
and ${\mathfrak A}^{\Theta} \in {\mathcal I}_N$
since ${\mathcal I}_N$ is a $\Z[G]$-module, whence (ii).
Notice that $(x) \in \Norm_{K/k}({\mathcal I}_N)$ expresses that $x$ is 
local norm in $K/k$ at every $v \notin S$; if moreover $x$
is local norm at $S$ it is a global one (Hasse norm theorem).
\end{proof}

In what follows, throughout the article, only the invariant 
$\rho = \rk_\Z(X)$ is needed (Definition \ref{Lambda} giving 
$\lambda := \max\,(0, \order S -1 - \rho)$) and never 
${\mathcal I}_N$, $\Lambda$,\,$\ldots$ However, $\rho = r + t$
may be unknown because of the minimal number $t$ of generators 
of the $\Z[G]$-module ${\mathcal I}_N$ (or the $\Z$-module ${\mathcal I}$) 
if it is too general; but, in practice, $\rho$ is known (e.g., 
$\rho = r_1+r_2-1$ for ray class groups modulo ${\mathfrak m}$)
or $\rho$ is large enough to get $\lambda = 0$.
 
\begin{lemma} \label{free}
Let $X = \langle x_1, \ldots , x_\rho \rangle_\Z$ be a free $\Z$-module 
of $\Z$-rank $\rho \geq 0$, let $\Omega := (\Z/p^n \Z)^{\delta}$, $n \geq 1$, 
$\delta \geq 0$, and let $f : X \to \Omega$ be an homomorphism 
such that the image of $X \to \Omega/\Omega^p$ 
is of $\F_p$-dimension $\min\,(\rho,\delta)$. Then $\order f(X) =  
p^{\min\,(\rho,\delta) \cdot n}$ and for all $m \in [0, n]$, $\order f^{p^m}(X) = 
p^{\min\,(\rho,\delta) \cdot (n-m)}$ (where $f^{p^m}(x) := f(x)^{p^m}$).
\end{lemma}

\begin{proof} 
From $f : X \to \Omega$, let $\ov f : X/X^p \to \Omega/\Omega^p$;
by assumption, $\dim_{\F_p}({\rm Im}(\ov f)) = \min\,(\rho,\delta)$.
Let $M = \langle e_1, \ldots , e_\delta \rangle_{\Z_p}$ be a free
$\Z_p$-module of $\Z_p$-rank $\delta$, and replace  $\Omega$ by $M/M^{p^n}$.
Let $q : X \to X/X^p$, $\pi_n : M \to M/M^{p^n}$, $\pi : M/M^{p^n} \to M/M^p$,
and $\pi_1 = \pi \circ \pi_n : M \to M/M^p$, be the canonical maps. Then, let
 $F : X \to M$ be any map such that $\pi_n \circ F = f$.
We have, since $n \geq 1$, the commutative diagram:
\unitlength=0.42cm
$$\vbox{\hbox{\hspace{-1.4cm}  
\begin{picture}(12.5,7.0)
\put(5,6.5){\line(1,0){2.0}}
\bezier{120}(6.8,6.4)(7,6.5)(7,6.5)
\bezier{120}(6.8,6.6)(7,6.5)(7,6.5)
\put(7.85,4.1){\line(0,1){1.9}}
\bezier{80}(7.95,4.2)(7.85,4)(7.85,4)
\bezier{80}(7.75,4.2)(7.85,4)(7.85,4)
\bezier{80}(7.95,4.4)(7.85,4.2)(7.85,4.2)
\bezier{80}(7.75,4.4)(7.85,4.2)(7.85,4.2)
\put(5,0.5){\line(1,0){2.0}}
\bezier{120}(6.8,0.4)(7,0.5)(7,0.5)
\bezier{120}(6.8,0.6)(7,0.5)(7,0.5)
\put(4.0,1){\line(0,1){5.0}}
\bezier{80}(4.1,1.2)(4,1)(4,1)
\bezier{80}(3.9,1.2)(4.0,1)(4.0,1)
\bezier{80}(4.1,1.4)(4,1.2)(4,1.2)
\bezier{80}(3.9,1.4)(4.0,1.2)(4.0,1.2)
\put(7.85,1){\line(0,1){1.9}}
\bezier{80}(7.95,1.2)(7.85,1)(7.85,1)
\bezier{80}(7.75,1.2)(7.85,1)(7.85,1)
\bezier{80}(7.95,1.4)(7.85,1.2)(7.85,1.2)
\bezier{80}(7.75,1.4)(7.85,1.2)(7.85,1.2)
\put(4.8,6.3){\line(1,-1){2.1}}
\put(6.475,4.0){$\searrow$}
\put(5.7,4.35){$\sst f$}
\put(2.8,6.3){\ft$X\!\! \simeq\! \Z^\rho$}
\put(3.0,0.3){\ft$X/X^p$}
\put(7.1,6.3){\ft$M \! \simeq\! \Z_p^\delta$}
\put(7.2,0.3){\ft$\Omega/\Omega^p \! \simeq\! M/M^p$}
\put(6.6,3.25){\ft$\Omega \! \simeq\! M/M^{p^n}$}
\put(5.7,6.67){$\sst F$}
\put(5.7,0.80){$\sst \ov f$}
\put(8.0,5){$\pi_n$}
\put(8.0,2){$\pi$}
\put(3.5,3.4){$q$}
\bezier{300}(9.8,6.3)(11.8,3.5)(9.8,0.7)
\put(10.9,3.4){$\pi_1$}
\end{picture}}}$$
\unitlength=1.0cm

(i) Case $\rho \geq \delta$ (surjectivity of $\ov f$). Let $y \in M$; there
exists $x \in X$ such that $\pi_1(y) = \ov f(q(x)) = \pi_1(F(x))$, whence
$y = F(x)\cdot y'^p$, $y' \in M$; so, by a finite induction, $M=F(X)\cdot M^{p^n}$
and $\pi_n(M) =  \Omega = f(X)$ is of order $p^{\delta \cdot n}$. 

\smallskip
(ii) Case $\rho < \delta$ (injectivity of $\ov f$). Let $y \in M$ such
that $y^p = F(x)$, $x \in X$; then $1 = \pi_1(y^p) = \pi_1 (F(x)) = \ov f(q(x))$,
whence $q(x)=1$ and $x = x'^p$, $x' \in X$, giving $y^p = F(x')^p$ in $M$ thus
$y = F(x')$, proving that $F(X)$ is direct factor in $M$; thus 
$f(X)= \pi_n(F(X))$ is a direct factor, in $\Omega$, isomorphic 
to $(\Z/p^n \Z)^\rho$.

\smallskip
Since $f^{p^m}(X) = f(X)^{p^m}$, this gives $f(X)^{p^m} = \Omega^{p^m}
\simeq ((\Z/p^n \Z)^\delta)^{p^m}$ in case (i) and $((\Z/p^n \Z)^\rho)^{p^m}$
in case (ii). Whence the orders.
\end{proof}

\section{Introduction of Hasse's norm symbols -- Main theorem}
Let $K/k$ be a $S$-ramified cyclic $p$-tower of degree $p^N$, $N \geq 1$. 
Let $\omega_{k_n/k}$ be the map which 
associates with $x \in \Lambda$ the family 
of Hasse's norm symbols $\big( \ffrac{x \, ,\, k_n/k}{v}\big) \in I_v(k_n/k)$
(inertia groups of $v \in S$), where $\Lambda := 
\{x \in k^\times,\ x \equiv 1 \pmod {\mathfrak m}, \, 
(x) \in \Norm_{K/k}({\mathcal I}_N)\}$. Since $x$ is local norm at 
the places $v \notin S$, the image of $\omega_{k_n/k}$ is contained in
$\Omega_{k_n/k} := \big \{ (\tau_v)_{v \in S} \in \bigoplus_{v \in S} 
I_v(k_n/k), \ \, \prod_{v \in S} \tau_v= 1 \big\}$ (from the product formula); 
then $\Ker(\omega_{k_n/k}) = \Lambda \cap \Norm_{k_n/k}(k_n^\times)$.  

\smallskip
Assume $S \ne \es$ totally ramified; so, fixing any $v_0 \in S$,
$\Omega_{k_n/k} \simeq \bigoplus_{v \ne v_0} G_n$, with 
$G_n := \Gal(k_n/k)$. We consider formula \eqref{totram} 
in $k_n/k$, using Lemma \ref{inclusion}\,(i): 
\begin{equation} \label{CG}
\order {\mathcal X}_n^{G_n} = \order  {\mathcal X} \cdot \ffrac{p^{n \,(\oorder S-1)}}
{(\Lambda : \Lambda \cap \Norm_{k_n/k}(k_n^\times))} = 
\order  {\mathcal X} \cdot \ffrac{p^{n \,(\oorder S-1)}}{\order\omega_{k_n/k}(\Lambda)},
\end{equation} 

\noindent
where ${\mathcal X}_n := {\mathcal C}_n/{\mathcal H}_n$ for the family 
$\{{\mathcal H}_n\}$
(cf. \eqref{DEF} with a prime-to-$S$ modulus ${\mathfrak m}$). 

\begin{theorem} \label{mt}
Set $\Lambda =: \tor_\Z(E) \plus X$,
with $X = \langle x_1, \ldots,  x_\rho \rangle_\Z$
free of $\Z$-rank $\rho$, and let $\lambda := \max\,(0, \order S -1 - \rho)$.
Assume that $\tor_\Z(E) \subset \Norm_{K/k}(K^\times)$ 
(see Lemma~\ref{albert} and Remark \ref{torsion}).
We then have the following properties of $\lambda$-stability:

\smallskip
(i) If $\order {\mathcal X}_1 = \order {\mathcal X} \!\cdot p^{\lambda}$,
then $\order {\mathcal X}_n = \order {\mathcal X} \cdot 
p^{\lambda \cdot n}$ and ${\mathcal X}_n^{G_n} = {\mathcal X}_n$,
for all $n \in [0, N]$.

\smallskip
(ii) If $\order {\mathcal X}_1 = \order {\mathcal X} \!\cdot\! p^{\lambda}$,
then $\J_{k_n/k}({\mathcal X}) = {\mathcal X}_n^{p^n}$ and
$\Ker(\J_{k_n/k}) = \Norm_{k_n/k}({\mathcal X}_n [p^n])$, for all $n \in [0,N]$, 
where the $\J_{k_n/k}$'s are the transfer maps in $k_n/k$ and 
${\mathcal X}_n [p^n] := \{x \in {\mathcal X}_n,\ x^{p^n}=1 \}$. 
If moreover, $\lambda = 0$, the norm maps $N_{k_n/k} : {\mathcal X}_n \to \!\!\!\!\! \to {\mathcal X}$ 
are isomorphisms.
\end{theorem}

\begin{proof}
Put $\Norm_n := \Norm_{k_n/k}$, $\J_n := \J_{k_n/k}$,
$G_n =: \langle \sigma_n \rangle$,
$\Omega_n := \Omega_{k_n/k}$, $\omega_n := \omega_{k_n/k}$, 
$s := \order S$. By assumption, $\omega_n (\Lambda)=\omega_n(X)$
for all $n \in [1,N]$. From \eqref{CG},
$\order {\mathcal X}_1^{G_1} = \order {\mathcal X} \cdot \ffrac{p^{s-1}}
{\order \omega_1(X)} \geq \order {\mathcal X} \cdot p^\lambda$
since $\order \omega_1(X) \leq p^{\min\,(\rho, s-1)}$; thus, 
with $\order {\mathcal X}_1 = \order {\mathcal X} \cdot p^{\lambda}$, 
we obtain ${\mathcal X}_1^{G_1} = {\mathcal X}_1$,
whence $\order {\mathcal X}_1^{G_1} = 
\order {\mathcal X} \cdot p^\lambda$ and $\order \omega_1(X) 
= p^{\min\,(\rho, s-1)}$.

\smallskip
By restriction of Hasse's symbols, $\omega_1 = \pi \circ \omega_n$,
with $\pi : \Omega_n \to \Omega_n/\Omega_n^p \simeq G_1^{s - 1}$.
Since we have proven that $\dim_{\F_p}(\omega_1(X) )= {\min\,(\rho, s-1)}$,
Lemma \ref{free} applies to $X=\langle x_1, \ldots, x_\rho \rangle_\Z$,
$\delta = s-1$, $f = \omega_n$; so, $\order \omega_n(X) = 
p^{\min\,(\rho, s-1)\, n}$ and, from \eqref{CG}:
\begin{equation}\label{Gn}
\order {\mathcal X}_n^{G_n} = \order {\mathcal X} \cdot p^{\lambda \cdot n},
 \  \hbox{for all $n \in [0, N]$} .
\end{equation}

Consider the extension $k_n/k_1$, of Galois group 
$G_n^p = \langle \sigma_n^p \rangle$, and the corresponding map
$\omega'_n$ on $k_1^\times$ with values in  $\Omega_{k_n/k_1} \simeq 
\Omega_n^p$; then $\order {\mathcal X}_n^{G_n^p} = \order {\mathcal X}_1\cdot 
\ffrac{p^{(n-1) \,(s-1)}}{\order \omega'_n (\Lambda_1)}$,
$\Lambda_1 := \{x \in k_1^\times,\, x \equiv 1\! \pmod {\mathfrak m},\, 
(x) \in \Norm_{K/k_1}({\mathcal I}_N)\}$. 
Since $\Lambda \subseteq \Lambda_1$ (Lemma \ref{inclusion}\,(ii)), 
the functorial properties of Hasse's symbols on $X$
imply (where $v_1 \mid v$ in $k_1$):
$$\omega'_n(x_i) = \Big(\Big(\ffrac{x_i, k_n/k_1}{v_1}\Big)\Big)_{\!v} = 
\Big(\Big(\ffrac{\Norm_1(x_i), k_n/k}{v}\Big)\Big)_{\!v} 
= \Big(\Big(\ffrac{x_i^p, k_n/k}{v}\Big)\Big)_{\!v}=
\Big(\Big(\ffrac{x_i, k_n/k}{v}\Big)^p\Big)_{\!v}, $$
giving $\omega'_n = \omega_n^p$ on $X$.
Then, $\omega'_n(X) \subseteq \omega'_n(\Lambda_1)$ yields:
$$\order {\mathcal X}_n^{G_n^p}  = \order {\mathcal X}_1\! \cdot\! 
\ffrac{p^{(n-1)\, (s-1)}}{\order \omega'_n(\Lambda_1)} \leq \order {\mathcal X}_1 
\! \cdot\!  \ffrac{p^{(n-1)\, (s-1)}}{\order \omega'_n(X)} =
\order {\mathcal X}_1\! \cdot\!  \ffrac{p^{(n-1)\, (s-1)}}{\order \omega^p_n(X)}; $$
Lemma \ref{free}, for $\omega_n^p$, gives 
$\order \omega^p_n(X) = p^{(n-1) \, \min\,(\rho, s-1)}$, then,
with $\order {\mathcal X}_1 = \order {\mathcal X} \cdot p^{\lambda}$:
$$\order {\mathcal X}_n^{G_n^p} \leq \order {\mathcal X}_1 
\! \cdot\!  \ffrac{p^{(n-1)\, (s-1)}}{p^{(n-1) \, \min\,(\rho, s-1)}}  
= \order {\mathcal X}_1 \! \cdot\!  p^{\lambda \cdot (n-1)} =  
\order {\mathcal X} \! \cdot\!  p^{\lambda \cdot n}.$$

Since ${\mathcal X}_n^{G_n^p} \supseteq 
{\mathcal X}_n^{G_n}$ and $\order {\mathcal X}_n^{G_n} = 
\order {\mathcal X} \!\cdot p^{\lambda \cdot n}$ from \eqref{Gn}, we get
${\mathcal X}_n^{G_n} = {\mathcal X}_n^{G_n^p}$, equivalent to
${\mathcal X}_n^{1-\sigma_n} = {\mathcal X}_n^{1-\sigma_n^p} = 
{\mathcal X}_n^{(1-\sigma_n) \,\cdot\, \theta}$,
where $\theta = 1+\sigma_n+ \cdots + \sigma_n^{p-1} \in (p, 1-\sigma_n)$,
a maximal ideal of $\Z_p[G_n]$ since $\Z_p[G_n]/(p, 1-\sigma_n) \simeq \F_p$; 
so ${\mathcal X}_n^{1-\sigma_n} = 1$, thus ${\mathcal X}_n={\mathcal X}_n^{G_n}$. 
Whence (i).

\smallskip
From $\Norm_n({\mathcal X}_n) = {\mathcal X}$, 
${\mathcal X}_n={\mathcal X}_n^{G_n}$ (from (i))
and $\J_n \circ \Norm_n = \nu_n$ (the algebraic norm), one obtains $\J_n({\mathcal X}) = 
\J_n(\Norm_n({\mathcal X}_n)) = {\mathcal X}_n^{\nu_n} = {\mathcal X}_n^{p^n}$. Let
$x \in \Ker(\J_n)$ and put $x = \Norm_n(y)$, $y \in  {\mathcal X}_n$; then 
$1 = \J_n(x) = \J_n(\Norm_n(y))=y^{p^n}$, so $ \Ker(\J_n) \subseteq 
\Norm_n({\mathcal X}_n[p^n])$ and if $x=\Norm_n(y)$, $y^{p^n}=1$,
then $\J_n(x)=\J_n(\Norm_n(y)) = y^{p^n}=1$. Whence (ii).
\end{proof}

\begin{corollary} \label{capitulation}
Let $p^e$ be the exponent of ${\mathcal X}$ in $k$ fixed, and let $K$ be 
a totally $S$-ramified cyclic $p$-tower of degree $p^N$, with $N \geq e$,
such that $\lambda=0$ (namely $1 \leq \order S \leq r_1+r_2$ for 
the family $\{{\mathcal C}_n\}$ of $p$-class groups). 
If $\order {\mathcal X}_1 = \order {\mathcal X}$, then ${\mathcal X}$ 
capitulates in $k_e$. In particular, this applies in a $\Z_p$-extension $\wt k/k$ 
with $\lambda=0$, $\order {\mathcal X}_1=\order {\mathcal X}$ and 
$S$ (set of $p$-places) totally ramified.
\end{corollary}

\begin{proof} Since $\lambda=0$, ${\mathcal X}_e \simeq {\mathcal X}$
from the isomorphism induced by $\Norm_e : {\mathcal X}_e \to \!\!\!\!\! \to {\mathcal X}$,
whence $\Ker(\J_e) = \Norm_e({\mathcal X}_e[p^e]) 
\simeq {\mathcal X}[p^e]={\mathcal X}$.
\end{proof}

\begin{example}
Put $p^\epsilon := \order \tor_{\Z_p}^{}(E)$.
For a prime $\ell \equiv 1\! \pmod{p^{N+\epsilon}}$, let $S_\ell$ be the set of prime ideals 
${\mathfrak l}$ of $k$ dividing $\ell$ and let $K \subseteq k(\mu_\ell)$ be the subfield
of degree $p^N$.
Sections \ref{nonmono} and \ref{examples} will give many examples of such 
capitulations of $p$-class groups of real fields $k$ in $K \subseteq k (\mu_\ell)$ 
(totally $S$-ramified with $S = S_\ell$); for instance, for $k=\Q(\sqrt {4409})$, 
the capitulation of the $3$-class group ${\mathcal C} \simeq \Z/9\Z$ occurs in 
$k_2$ for $\ell \in$ $\{19$, $37$, $73$, $109$, $127$, $271$, 
$307$, $379$, $397$, $523$, $541$, $577$, $739$, $883$, $\,\ldots \}$.
The $3$-class group ${\mathcal C} \simeq \Z/9\Z \times \Z/3\Z$ of the cyclic
cubic field of conductor $5383$, defined by $P=x^3+x^2-1794 x+17744$,
capitulates in $k_2$ for $\ell = 109, 163, 919,\,\ldots$
The $2$-class group ${\mathcal C} \simeq \Z/16\Z$ of $\Q(\sqrt{2305})$ 
capitulates in in $k_4$ for $\ell = 97, 193, 353, 449, 929,\,\ldots$ 
The $5$-class group ${\mathcal C} \simeq \Z/25\Z$ of $\Q(\sqrt{24859})$ 
capitulates in in $k_2$ for $\ell = 101, 151, 251, 401,\,\ldots$

\smallskip
Totally real fields $k$ give $\lambda = 0$ whatever $S = S_\ell$, but some 
non-totally real fields may give $\lambda = 0$; for instance, let 
$P=x^6 + x^5 - 5 x^4 - 4 x^3 + 6 x^2 + 2 x + 7$
defining $k$, of signature $(0,3)$, of Galois group ${\rm S}_6$
and discriminant $-11\cdot 31 \cdot 971 \cdot 2801$;
so $\lambda = \max (0,\order S -1 - (r_1+r_2-1)) = 
\max(0, \order S - 3)$. For $p=3$, 
${\mathcal C} \simeq \Z/3\Z$ capitulates in $k_1$ for $\ell = 37,
61, 67, 73, 97, 103, 109, \,\ldots$ for which
$\order S_\ell \in \{1,2,3\}$, giving $\lambda = 0$ 
as expected. We have verified the capitulation of ${\mathcal C}$ in $k_1$ 
for $\ell = 37$ (for this and other complements see \cite{Gra+}).  
\end{example}

\section{On the existence of capitulation fields for real class groups}\label{nonmono}
A totally real field $k$ been given, a problem is {\it the existence} of cyclic $p$-towers 
$K/k$ such that ${\mathcal C}$, of exponent $p^e$, capitulates in $K$.
We do not intend to establish again the ``abelian capitulation'' proved in the literature 
(Gras \cite{Gcap} (1997), Kurihara \cite{Kcap} (1999), Bosca \cite{Bcap} (2009), 
Jaulent \cite{Jau2,Jcap} (2019, 2020)), but we will examine this possibility in a 
simpler way using Theorem \ref{mt}. We may assume that $k$ is Galois real and 
we shall make some heuristics using only cyclic $p$-towers totally $S$-ramified 
with $S=S_\ell$, the set of places of $k$ above $\ell$; this is an important difference 
compared to the previous references.

\subsection{Monogenic class groups and $K \subset k(\mu_\ell)$}
We assume, at first, that ${\mathcal C}$ is monogenic, that is to say that
there exists a prime ideal ${\mathfrak q}$ such that ${\mathcal C}$ is
generated by the classes of the conjugates of ${\mathfrak q}$.
Let $K \subseteq k(\mu_\ell)$, $\ell \equiv 1 \pmod {2p^N}$, $N \gg 0$, be a 
cyclic $p$-tower of degree $p^N$ and let $S=S_\ell$. We assume $\ell$
totally split in $k$, whence $\order S = [k : \Q].$
We have, for $k_e/k$ of Galois group $G_e$, the exact sequence:
$1 \to \cl_e({\mathcal J}_e^{G_e}) \too {\mathcal C}_e^{G_e} 
\too E \cap \Norm_e(k_e^\times)/ \Norm_e(E_e) \to 1$,
where the group ${\mathcal J}_e^{G_e}$, of invariant ideals 
of $k_e$, is of the form $\langle S_e \rangle \cdot \J_e ({\mathcal J})$, where
$S_e$ is the set of prime ideals of $k_e$ above $S$ and $\J_e ({\mathcal J})$ the 
extension to $k_e$ of the group of ideals of $k$. 
From Theorem \ref{mt}, necessary conditions, to get the criterion of stability
$\order {\mathcal C}_1 = \order{\mathcal C}$, are ${\mathcal C}_e = 
{\mathcal C}_e^{G_e}$ and $(E : E \cap \Norm_e(k_e^\times)) = 
p^{e (\oorder S - 1)} = p^{e ([k : \Q] - 1)}$, whence 
$E \cap \Norm_e(k_e^\times) = E^{p^e} \subseteq \Norm_e(E_e)$ giving 
${\mathcal C}_e^{G_e} = \cl_e(\langle S_e \rangle) \cdot  \J_e ({\mathcal C})$. 
But the condition ${\mathcal C}_e ={\mathcal C}_e^{G_e}$
is equivalent to $({\mathcal C}_e/{\mathcal C}_e^{G_e})^{G_e}=1$,
whence (from \eqref{totram}), to
$\ds \frac{\order {\mathcal C}}{\order \Norm_e({\mathcal C}_e^{G_e})} \times 
\frac{p^{e ([k\, :\, \Q] - 1)}}{(\Lambda' : \Lambda' \cap  \Norm_e (k_e^\times))} = 1$, 
for some $\Lambda' \supseteq E$ so that the second factor is trivial.
Since $\Norm_e(\langle S_e \rangle) = \langle S \rangle$, 
$\ds \ds \frac{\order {\mathcal C}}{\order \Norm_e({\mathcal C}_e^{G_e})}
= \frac{\order {\mathcal C}}{\order \cl(\langle S \rangle) \cdot \order {\mathcal C}^{p^e}}
= \ds \frac{\order {\mathcal C}}{\order \cl(\langle S \rangle)} = 1$ if and only if
$S$ generates ${\mathcal C}$. 
Under the assumption of monogenicity, from Chebotarev density theorem, 
there exist infinitely many such sets and an obvious heuristic is that among 
these sets, infinitely many ones give the criterion $\order {\mathcal C}_1 
= \order {\mathcal C}$, hence capitulation of ${\mathcal C}$.
But we can go further for non-monogenic ${\mathcal C}$'s.

\subsection{Non-monogenic class groups and $K \subset k(\mu_\ell)$}
Using only single primes $\ell$, the condition $\order {\mathcal C}_1 = 
\order{\mathcal C}$ does not hold in general in the first layer $k_1$. 
Nevertheless, many examples lead 
to a stabilization from some $n_0>1$, hence capitulation of ${\mathcal C}$ 
from $k'_e = k_{n_0+e}$, so that we must have $N > n_0+e$, where $e$ is 
a constant, which suggests the existence of such towers.
Unfortunately, any verification becomes out of reach if $n_0$ is too large. 
For instance, using the program of \S\,\ref{quad} or the program below,
for $k=\Q(\sqrt{130})$, the $p$-class group ${\mathcal C} \simeq (\Z/2\Z)^2$ 
capitulates in $k_2$ for $\ell = 1697$, $2017$, $5153$, $5857,\,\ldots$.
For $k=\Q(\sqrt {2310})$, where 
${\mathcal C} \simeq (\Z/2\Z)^3$, we have the following results:\par
\ft\begin{verbatim}
ell=593                     ell=1217                    ell=4289
v0=3 C0=[2,2,2]            v0=3 C0=[2,2,2]            v0=3  C0=[2,2,2]
v1=6 C1=[2,4,8]            v1=6 C1=[2,4,8]            v1=7  C1=[2,2,2,16]
v2=7 C2=[2,4,16]           v2=7 C2=[2,4,16]           v2=9  C2=[2,2,4,32]
v3=7 C3=[2,4,16]           v3=7 C3=[2,4,16]           v3=10 C3=[2,2,4,64]
\end{verbatim}\ns

For $\ell = 593$ (resp. $\ell = 1217$) the stability holds from $n_0=2$ 
with $N=3$ (resp. $N=5$); thus ${\mathcal C}_2[2]$ capitulates in 
$k'_1=k_3$, whence ${\mathcal C}[2] = {\mathcal C}$ capitulates in $k_3$.
For $\ell=4289$, a stabilization seems possible from $n_0=3$, but
${\mathcal C}_4$ is not computable.

\smallskip
Concerning capitulation of ${\mathcal C}$ in towers $K \subset k(\mu_\ell)$, 
we have the following experiments, showing that it can happen even 
if no stabilization is obtained; the notation ${\sf Cn=[A,\ldots, Z]}$ means 
${\mathcal C}_n \simeq \Z/A\Z \times \ldots \times \Z/Z \Z$ in $k_n$ and 
a box ${\sf [a,\ldots ,z]}$ means that the generators of ${\mathcal C}$
extended in $k_n$ are, respectively, the $a$th, $\ldots$, $z$th powers of
suitable generators of ${\mathcal C}_n$ computed by PARI/GP 
($a,\ldots ,z$, being integers modulo $A,\ldots, Z$, respectively); so 
that ${\sf [0, \ldots, 0]}$ is equivalent to the capitulation of ${\mathcal C}$.
For the program, one must precise ${\sf p}$ the minimal $p$-rank of 
${\mathcal C}$ required ${\sf rpmin}$, the length ${\sf N}$ of the tower 
and the interval for ${\sf m}$:\par
\ft\begin{verbatim}
{p=2;rpmin=3;N=3;bm=2;Bm=10^6;for(m=bm,Bm,if(core(m)!=m,next);
P=x^2-m;k=bnfinit(P,1);Ck=k.clgp;r=matsize(Ck[2])[2];\\Ck=class group of k
L=List;for(i=1,r,listput(L,0,i));rp=0;for(i=1,r,ei=Ck[2][i];v=valuation(ei,p);
if(v>0,rp=rp+1));if(rp<rpmin,next);\\computation and test of the p-rank rp
h=k.no;u=h/p^valuation(h,p);Ckp=List;for(i=1,r,ai=idealpow(k,Ck[3][i],u);
listput(Ckp,ai,i));\\representatives ai of the p-class group Ckp
forprime(ell=5,200,if(Mod(ell-1,2*p^N)!=0||Mod(m,ell)==0,next);Lq=List;
\\the program computes prime representatives qi, inert in K/k, split in k:
for(i=1,r,ai=Ckp[i];forprime(q=2,10^4,if(q==ell || kronecker(m,q)!=1,next);
o=znorder(Mod(q,ell));if(valuation(o,p)!=valuation(ell-1,p),next);\\inertia
f=idealfactor(k,q);qi=component(f,1)[1];cij=ai;for(j=1,p-1,cij=idealmul(k,cij,qi);
if(List(bnfisprincipal(k,cij)[1])==L,listput(Lq,q,i);break(2)))));
print("m=",m," Lq=",Lq);\\Lq = list of primes qi generating the p-class group
for(n=0,N,R=polcompositum(P,polsubcyclo(ell,p^n))[1];K=bnfinit(R,1);H=K.no;
U=H/p^valuation(H,p);print();print("ell=",ell," n=",n," CK",n,"=",K.cyc);
for(i=1,r,Fi=idealfactor(K,Lq[i]);Qi=idealpow(K,component(Fi,1)[1],U);
print(bnfisprincipal(K,Qi)[1])))))}
\end{verbatim}\ns
$p=2$
\ft\begin{verbatim}
m=1155  ell=193
C0=[2,2,2]      C1=[2,2,2,2,2]      C2=[4,4,2,2,2,2,2]      C3=[8,8,4,2,2,2,2]
[1,0,0]         [1,0,1,0,0]         [0,0,0,0,0,0,0]         [0,0,0,0,0,0,0] 
[0,1,0]         [1,1,1,0,0]         [2,2,1,1,1,0,0]         [0,0,0,0,0,0,0]            
[0,0,1]         [0,1,0,0,0]         [2,2,1,1,1,0,0]         [0,0,0,0,0,0,0]
\end{verbatim}\ns
\ft\begin{verbatim}
m=1155  ell=337
C0=[2,2,2]      C1=[2,2,2,2,2]      C2=[20,4,4,2,2]         C3=[40,4,4,2,2]
[1,0,0]         [0,1,1,1,1]         [0,0,0,0,0]             [0,0,0,0,0]
[0,1,0]         [1,0,1,1,0]         [10,2,0,0,0]            [0,0,0,0,0]
[0,0,1]         [1,1,0,0,1]         [10,2,0,0,0]            [0,0,0,0,0]
\end{verbatim}\ns
\ft\begin{verbatim}
m=1995  ell=17 
C0=[2,2,2]      C1=[4,2,2,2]        C2=[8,4,2,2]            C3=[8,4,2,2]
[1,0,0]         [2,1,0,0]           [0,0,0,0]               [0,0,0,0]
[0,1,0]         [2,1,0,0]           [0,0,0,0]               [0,0,0,0]
[0,0,1]         [0,1,0,0]           [4,0,0,0]               [0,0,0,0]
\end{verbatim}\ns
\ft\begin{verbatim}
m=1995  ell=113
C0=[2,2,2]      C1=[12,2,2,2]       C2=[12,2,2,2,2,2]       C3=[24,4,2,2,2,2]
[1,0,0]         [6,1,1,1]           [0,0,1,1,1,0]           [0,0,0,0,0,0]
[0,1,0]         [6,0,0,0]           [0,0,0,0,0,0]           [0,0,0,0,0,0]
[0,0,1]         [0,0,0,0]           [0,0,0,0,0,0]           [0,0,0,0,0,0] 
\end{verbatim}\ns
\ft\begin{verbatim}
m=1995  ell=193 
C0=[2,2,2]      C1=[2,2,2,2,2]      C2=[12,6,2,2,2,2,2]     C3=[24,12,4,4,2,2,2]
[1,0,0]         [1,0,1,1,0]         [6,3,1,0,0,0,0]         [0,0,0,0,0,0,0]
[0,1,0]         [0,1,1,1,0]         [6,3,1,0,0,0,0]         [0,0,0,0,0,0,0]
[0,0,1]         [0,0,0,0,0]         [0,0,0,0,0,0,0]         [0,0,0,0,0,0,0]
\end{verbatim}\ns
\ft\begin{verbatim}
m=2210  ell=97 
C0=[2,2,2]      C1=[12,4,2]         C2=[24,4,2]             C3=[24,4,2]
[1,0,0]         [0,0,0]             [0,0,0]                 [0,0,0]
[0,1,0]         [6,0,0]             [12,0,0]                [0,0,0]
[0,0,1]         [6,2,0]             [12,0,0]                [0,0,0]
\end{verbatim}\ns
\ft\begin{verbatim}
m=2210  ell=113 
C0=[2,2,2]      C1=[12,2,2,2]       C2=[24,4,2,2,2]         C3=[48,12,12,4,2]
[1,0,0]         [0,1,1,0]           [0,2,0,0,0]             [0,0,0,0,0]
[0,1,0]         [6,0,0,0]           [12,0,0,0,0]            [24,0,0,0,0]
[0,0,1]         [6,1,1,0]           [12,2,0,0,0]            [24,0,0,0,0]
\end{verbatim}\ns
$p=3$
\ft\begin{verbatim}
m=23659  ell=19             m=23659 ell=37                  m=32009 ell=19
C0=[6,3]C1=[18,3]C2=[18,3]  C0=[6,3]C1=[18,3,3]C2=[18,3,3]  C0=[3,3]C1=[9,3]C2=[9,3]
[0,1]   [0,0]    [0,0]      [0,2]   [12 0,0]   [0,0,0]      [0,1]   [6,0]   [0,0]
[2,0]   [6,0]    [0,0]      [2, 0]  [6,0,0]    [0,0,0]      [1,2]   [6,0]   [0,0]
\end{verbatim}\ns
\ft\begin{verbatim}
m=32009 ell=37                      m=42817 ell=19
C0=[3,3]  C1=[9,3]   C2=[9,3]       C0=[3,3]  C1=[9,3]  C2=[27,3]
[0,1]     [6,0]      [0,0]          [2,1]     [6,0]     [9,0]
[1,2]     [3,0]      [0,0]          [1,0]     [3,0]     [18,0]
\end{verbatim}\ns

In the non-monogenic case, we observe many capitulations (or partial ones), in cyclic 
$p$-tower contained in $k(\mu_\ell)$ (possibly, the capitulation holds in larger layers 
but this would require several days of computer); so we propose the following conjecture:

\begin{conjecture}\label{conjcap}
Let $k$ be a totally real number field with generalized $p$-class group 
${\mathcal X}_{k,{\mathfrak m}} = {\mathcal C}_{k,{\mathfrak m}}/{\mathcal H}_k$ 
(cf. \eqref{DEF}). There are infinitely many primes $\ell$, $\ell \equiv 1 \pmod {2\, p^{N}}$, 
$N \gg 0$, such that ${\mathcal X}_{k,{\mathfrak m}}$ capitulates in $k(\mu_\ell)$.
\end{conjecture}

Examples of non-totally real fields (case of the degree $6$ example and non-Galois 
cubic fields, as illustrated in \cite{Gra+}) show that the property can be enlarged to 
some non-totally real fields, with some conditions on the infinite places since it is 
obvious that capitulation in $k(\mu_\ell)$ does not hold for an imaginary quadratic field.
Since the structures of these $p$-class groups depend on norm properties, the
conjecture can probably be also completed by densities results 
thanks to the techniques used in the classical approaches \cite{Bcap,Gcap,Jcap,Kcap} 
and especially those of \cite{KP,Sm}.

\section{Examples and counterexamples, with modulus, of $\lambda$-stabilities}\label{examples}
\subsection{Cyclic $p$-towers over $k = \Q(\sqrt{\pm m})$}\label{quad}

Consider, for $\ell \equiv 1 \pmod{p^N}$ (in ${\sf ell}$), the 
cyclic $p$-tower contained in $k (\mu_\ell)/k$, for $k=\Q(\sqrt{s\,m})$, 
$m >0$ square free (in ${\sf m}$), $s = \pm 1$. The following PARI/GP 
program computes ${\mathcal C}_{k_n,{\mathfrak m}} =: {\mathcal C}_n$ 
for any prime-to-$\ell$ modulus ${\mathfrak m}$ (in ${\sf mod}$). 
The polynomial $P_n$ (in ${\sf Pn}$) defines the layer $k_n$.
If $s=-1$ (resp. $s=1$), $\lambda = \order S -1 - \rho \in\{0,1\}$ (resp. $\lambda = 0$)
in $k$ but decreases in the tower (Remark \ref{lambda=0}). 
Then ${\sf vn}$ denotes the $p$-valuation of the order of ${\mathcal C}_n$ whose 
structure ${\sf Cn}$ is given by the list ${\sf C}$.

\smallskip
The computations up to ${\sf N} = 3$ or $4$ (${\sf n \in [0,N]}$) are only for verification when 
a $\lambda$-stability exists; otherwise some examples do not $\lambda$-stabilize 
in the interval considered and no conclusion is possible.
For the program, one must precise ${\sf p}$, ${\sf ell}$, the modulus ${\sf mod}$,
${\sf s=\pm1}$ defining real or imaginary quadratic fields $k$,
the length ${\sf N}$ of the tower and the interval for ${\sf m}$:\par
\ft\begin{verbatim}
{p=2;ell=257;mod=5;s=-1;N=3;bm=2;Bm=10^3;
for(m=bm,Bm,if(core(m)!=m || Mod(m,ell)==0,next);
P=x^2-s*m;lambda=0;if(s==-1&kronecker(s*m,ell)==1,lambda=1);\\lambda 
print();print("p=",p," mod=",mod," ell=",ell," sm=",s*m," lambda=",lambda);
for(n=0,N,Pn=polcompositum(polsubcyclo(ell,p^n),P)[1];kn=bnfinit(Pn,1);\\layer kn
knmod=bnrinit(kn,mod);v=valuation(knmod.no,p);Cn=knmod.cyc;\\ray class group of kn
C=List;d=matsize(Cn)[2];for(j=1,d,c=Cn[d-j+1];w=valuation(c,p);
if(w>0,listinsert(C,p^w,1)));\\end of computation of the p-ray class group of kn
print("v",n,"=",v," p-ray class group=",C)))}
\end{verbatim}\ns
\ft\begin{verbatim}
IMAGINARY QUADRATIC FIELDS, p=2, ell=257, mod=1:
m=-2,lambda=1   m=-11,lambda=1  m=-14,lambda=0   m=-17,lambda=1        m=-15,lambda=1  
v0=0  []        v0=0  []        v0=2  [4]        v0=2 [4]              v0=1 [2]    
v1=3  [8]       v1=2  [4]       v1=4  [4,4]      v1=5 [8,2,2]          v1=2 [4] 
v2=7  [16,4,2]  v2=7  [8,8,2]   v2=8  [4,4,4,4]  v2=10[16,2,2,2,2,2,2] v2=3 [8]  
v3=13 [32,8,2,  v3=10 [16,16,4] v3=16 [8,8,8,4,  v3=19[32,2,2,2,2,2,   v3=4 [16]
       2,2,2,2]                         4,2,2,2]    2,2,2,2,2,2,2,2,2]
\end{verbatim}\ns
\ft\begin{verbatim}
m=-782,lambda=1                               m=-858,lambda=0
v0=3 [4,2]                                    v0=4  [4,2,2]
v1=8 [4,4,4,2,2]                              v1=12 [128,2,2,2,2,2]
v2=22[32,32,4,4,4,2,2,2,2,2,2]                v2=22 [256,16,8,2,2,2,2,2,2,2]
v3=35[64,64,8,8,8,2,2,2,2,2,2,2,2,2,2,2,2,2,2]v3=32 [512,32,16,4,4,4,2,2,2,2,2,2,2,2]
\end{verbatim}\ns
\ft\begin{verbatim}
m=-15,mod=5,lambda=1  m=-35,mod=5,lambda=1 m=-253,mod=5,lambda=1
v0=2 [4]              v0=2 [4]             v0=4  [4,2,2]
v1=3 [4,2]            v1=3 [4,2]           v1=12 [64,4,2,2,2,2]
v2=4 [8,2]            v2=4 [8,2]           v2=22 [128,8,4,4,2,2,2,2,2,2,2,2]
                                           v3=34 [256,16,8,4,4,4,4,4,4,2,2,2,2,2,2,2]
\end{verbatim}\ns
\ft\begin{verbatim}
IMAGINARY QUADRATIC FIELDS, p=3, ell=163, mod=1:
m=-2,lambda=1       m=-293,lambda=1     m=-983,lambda=1       m=-3671,lambda=0
v0=0 []             v0=2 [9]            v0=3 [27]             v0=4  [81]
v1=1 [3]            v1=3 [27]           v1=4 [81]             v1=7  [243,3,3]
v2=2 [9]            v2=4 [81]           v2=5 [243]            v2=11 [729,9,9,3]
\end{verbatim}\ns
\ft\begin{verbatim}
REAL QUADRATIC FIELDS, p=2, lambda=0:
mod=7,ell=12289:
m=2          m=19           m=89        m=15                m=39
v0=0 []      v0=1 [2]       v0=0 []     v0=2 [2,2]          v0=2  [2,2]
v1=1 [2]     v1=2 [2,2]     v1=0 []     v1=5 [2,2,2,2,2]    v1=5  [2,2,2,2,2]
v2=1 [2]     v2=2 [2,2]     v2=0 []     v2=9 [8,4,4,2,2]    v2=10 [8,4,2,2,2,2,2]
\end{verbatim}\ns
\ft\begin{verbatim}
mod=17,ell=7340033
m=17      m=19          m=21        m=38          m=26             m=226
v0=2 [4]  v0=6 [16,4]   v0=4 [16]   v0=5 [16,2]   v0=6 [16,2,2]    v0=5  [8,8]
v1=2 [4]  v1=6 [16,4]   v1=4 [16]   v1=5 [16,2]   v1=8 [16,4,2,2]  v1=10 [32,8,8]
\end{verbatim}\ns
\ft\begin{verbatim}
REAL QUADRATIC FIELDS, p=3, lambda=0:
mod=1,ell=109           mod=1,ell=109        mod=7,ell=163          mod=7,ell=163
m=326                   m=4409               m=5                    m=6 
v0=1 [3]                v0=2 [9]             v0=1 [3]               v0=1 [3]
v1=1 [3]                v1=2 [9]             v1=1 [3]               v1=2 [3,3]
v2=1 [3]                v2=2 [9]             v2=1 [3]               v2=2 [3,3]
\end{verbatim}\ns
\ft\begin{verbatim}
mod=7,ell=163           mod=7,ell=163        mod=19,ell=163         mod=19,ell=163
m=22                    m=15                 m=2                    m=11 
v0=2 [3,3]              v0=1 [3]             v0=2 [9]               v0=3 [9,3]
v1=2 [3,3]              v1=3 [3,3,3]         v1=3 [27]              v1=4 [27,3]
v2=2 [3,3]              v2=5 [9,9,3]         v2=3 [27]              v2=4 [27,3]
\end{verbatim}\ns

One finds examples of $\lambda$-stabilities taking for instance $\order S = 6$ 
with three primes split in $k=\Q(\sqrt {\pm m})$ and any of the $1024$ 
totally $S$-ramified cyclic $2$-towers $K/k$ of degree 
$2^6$ contained in $k (\mu_{257 \cdot 449 \cdot 577})$ but, unfortunately,  
only computing in $k$ and $k_1 = k(\sqrt {257 \cdot 449 \cdot 577})$ with
$\lambda = 4$ (resp.~$5$) for $k$ real (resp. imaginary):\par
\ft\begin{verbatim}
m=193          m=1591           m=2669              m=-527           m=-1739 
v0=0 []        v0=1 [2]         v0=2 [4]            v0=1 [2]         v0=2 [4]
v1=4 [2,2,2,2] v1=5 [2,2,2,2,2] v1=6 [4,2,2,2,2]    v1=6 [4,2,2,2,2] v1=7 [8,2,2,2,2]
\end{verbatim}\ns

\subsection{Cyclic $p$-towers over $\Q$ with modulus}
In a cyclic $p$-tower $K/\Q$ for $S=S_\ell$, ${\mathfrak m}=1$, \eqref{totram} 
shows that $\lambda=0$ and $\order {\mathcal C}_n=1$ for all $n$. 
For ${\mathfrak m} \ne 1$, a stability may occur (possibly from $n_0>0$) 
as shown by the following examples ($p=2, 3$) analogous to \cite[Example 4.1]{MY}:\par
\ft\begin{verbatim}
{p=2;ell=257;mod=17;N=5;print("p=",p," ell=",ell," mod=",mod);
for(n=0,N,Pn=polsubcyclo(ell,p^n);kn=bnfinit(Pn,1);\\layer kn
knmod=bnrinit(kn,mod);v=valuation(knmod.no,p);Cn=knmod.cyc;\\ray class group
C=List;d=matsize(Cn)[2];for(j=1,d,c=Cn[d-j+1];w=valuation(c,p);
if(w>0,listinsert(C,p^w,1)));print("v",n,"=",v," C",n,"=",C))}
\end{verbatim}\ns
\ft\begin{verbatim}
p=2,ell=257,mod=17   ell=257,mod=4*17          p=3,ell=163,mod=19   ell=163,mod=9*19
v0=3 [8]             v0=4  [16]                v0=2 [9]             v0=3 [9,3]
v1=4 [16]            v1=5  [16,2]              v1=3 [27]            v1=5 [27,3,3]
v2=4 [16]            v2=7  [16,2,2,2]          v2=3 [27]            v2=6 [81,3,3]
v3=4 [16]            v3=11 [16,2,2,2,2,2,2,2]  v3=3 [27]            v3=6 [81,3,3]
v4=4 [16]            v4=12 [32,2,2,2,2,2,2,2]  
v5=4 [16]            v5=12 [32,2,2,2,2,2,2,2]  
\end{verbatim}\ns

\subsection{Cyclic $p$-towers over a cyclic cubic field $k$}
Let $K \subset k (\mu_\ell)$ defining a cyclic $p$-tower of $k$. The program 
gives the complete list of cyclic cubic fields of conductor ${\sf f \in [bf, Bf]}$.
An interesting fact is that, in most cases, the stability occurs 
(with $\lambda = 0$ in the case $S=S_\ell$); we give an excerpt with the 
defining polynomial ${\sf P}$ of $k$ of conductor ${\sf f}$:\par
\ft\begin{verbatim}
{p=2;ell=257;N=3;bf=7;Bf=10^3;for(f=bf,Bf,if(Mod(f,ell)==0,next);
h=valuation(f,3);if(h!=0 & h!=2,next);F=f/3^h;if(core(F)!=F,next);F=factor(F);
Div=component(F,1);d=matsize(F)[1];for(j=1,d,D=Div[j];if(Mod(D,3)!=1,break));
for(b=1,sqrt(4*f/27),if(h==2 & Mod(b,3)==0,next);A=4*f-27*b^2;
if(issquare(A,&a)==1,\\computation of a and b such that f=(a^2+27b^2)/4
if(h==0,if(Mod(a,3)==1,a=-a);P=x^3+x^2+(1-f)/3*x+(f*(a-3)+1)/27);
if(h==2,if(Mod(a,9)==3,a=-a);P=x^3-f/3*x-f*a/27);print();print("f=",f," P=",P);
\\end of computation of P defining k of conductor f
for(n=0,N,Pn=polcompositum(polsubcyclo(ell,p^n),P)[1];kn=bnfinit(Pn,1);
v=valuation(kn.no,p);Cn=kn.cyc;C=List;d=matsize(Cn)[2];for(j=1,d,c=Cn[d-j+1];
w=valuation(c,p);if(w>0,listinsert(C,p^w,1)));print("v",n,"=",v," C",n,"=",C)))))}
\end{verbatim}\ns
\ft\begin{verbatim}
p=2,ell=257,lambda=0
f=63               f=163                  f=277                   f=279
P=x^3-21*x-35      P=x^3+x^2-54*x-169     P=x^3+x^2-92*x+236      P=x^3-93*x+217
v0=0 []            v0=2 [2,2]             v0=2 [2,2]              v0=0 []
v1=2 [2,2]         v1=4 [4,4]             v1=4 [2,2,2,2]          v1=2 [2,2]
v2=2 [2,2]         v2=4 [4,4]             v2=6 [2,2,2,2,2,2]      v2=4 [4,4]
v3=2 [2,2]         v3=4 [4,4]             v3=6 [2,2,2,2,2,2]            ?        
\end{verbatim}\ns
\ft\begin{verbatim}
f=333              f=349                  f=397                   f=547 
P=x^3-111*x+37     P=x^3+x^2-116*x-517    P=x^3+x^2-132*x-544     P=x^3+x^2-182*x-81
v0=0 []            v0=2 [2,2]             v0=2 [2,2]              v0=2 [2,2]
v1=4 [4,4]         v1=4 [2,2,2,2]         v1=2 [2,2]              v1=4 [2,2,2,2]
v2=4 [4,4]         v2=4 [2,2,2,2]         v2=2 [2,2]              v2=6 [4,4,2,2]
\end{verbatim}\ns
\ft\begin{verbatim}
p=3,ell=109,lambda=0
f=1687               f=1897               f=2359                f=5383
x^3+x^2-562*x-4936   x^3+x^2-632*x-4075   x^3+x^2-786*x+7776    x^3+x^2-1794*x+17744
v0=1 [3]             v0=1 [3]             v0=2 [3,3]            v0=3 [9,3]
v1=1 [3]             v1=1 [3]             v1=2 [3,3]            v1=3 [9,3]
\end{verbatim}\ns

As for the quadratic case, we may use $k_1 = k(\sqrt {257 \cdot 449 \cdot 577})$ 
for $p=2$ with $\order S = 9$, $\rho = 2$, $\lambda = 6$. The contribution of the 
$2$-tower over $\Q$ is given by the $2$-stability from the structures 
${\sf [\ ], [2,2]}$. We find a great lot of $\lambda$-stabilities with the structures 
${\sf [\ ], [2,2,2,2,2,2]}$ or ${\sf [2,2], [4,4,2,2,2,2]}$ and very few exceptions 
for which one does not know if a $\lambda$-stabilization does exist for $n>1$ 
(case of ${\sf f=1531}$ below):\par
\ft\begin{verbatim}
f=171,P=x^3-57*x-152    f=349,P=x^3+x^2-116*x-517    f=1531,P=x^3+x^2-510*x-567 
C0=[]                   C0=[2,2]                     C0=[]
C1=[2,2,2,2,2,2]        C1=[4,4,2,2,2,2]             C1=[4,4,2,2,2,2]
\end{verbatim}\ns

\section{Other examples of applications and Remarks}
Assume the total ramification of $S \ne \es$ in the cyclic $p$-tower $K/k$ and the main
condition $\order{\mathcal X}_1 = \order {\mathcal X} \cdot p^\lambda$ with
$\lambda = \max\,(0, \order S - 1 - \rho)$ (Definition \ref{Lambda}); remember 
that it is easy to ensure that $\tor_\Z(E) \subset \Norm_{K/k}(K^\times)$ 
(Remark \ref{torsion}):

\smallskip
({\bf a}) Let $\wt k/k$ be a $\Z_p$-extension with Iwasawa's invariants 
$\wt \lambda$, $\wt \mu$, $\wt \nu$. We then have a set 
of $p$-places $S$ with $1 \leq \order S \leq [k : \Q] = r_1+2\,r_2$ and 
$\lambda = \max\,(0, \order S - r_1 - r_2) \in [0, r_2]$; one gets the formula 
$\order{\mathcal X}_n = \order{\mathcal X} \cdot p^{\lambda \cdot n}$ 
for all $n$, as soon as $\order{\mathcal X}_1 = 
\order {\mathcal X} \cdot p^\lambda$ holds, and in that case, $\wt \lambda$ is 
equal to $\lambda$ and $\wt \mu=0$; the Iwasawa formula being 
$\order{\mathcal X}_n = p^{\lambda \cdot n + \nu}$ with $p^{\nu} := 
\order{\mathcal X}$ (recall that this may apply only from a larger layer $k'$).
For CM-fields $k$ with $p$ totally split in $k$, $\lambda = \frac{1}{2}\,[k : \Q]$.

\smallskip
When $k$ is totally real, Greenberg's conjecture \cite{Gre} for $p$-class groups 
($\wt \lambda = \wt \mu = 0$) 
holds if and only if $\order{\mathcal C}_{n_0+1} = \order {\mathcal C}_{n_0}$
is fulfilled for some $n_0 \geq 0$ ($\lambda'$-stabilization from $k' = k_{n_0}$
with $\lambda' = \wt \lambda = 0$), which may be checked if $n_0$ is not too 
large by chance (as in \cite{KS} for real quadratic fields, 
$p=3$ non-split, or in various works of Taya \cite{Ta}, Fukuda, Komatsu).

({\bf b}) Taking ${\mathcal H}_N = {\mathcal C}_N^{p^r}$ may give the 
$\lambda(r)$-stability of the ${\mathcal C}_n/{\mathcal C}_n^{p^r}$'s,
for some $\lambda(r)$, whence $\rk_p({\mathcal C}_n) = 
\rk_p({\mathcal C}) + \lambda(1) \cdot n$ for all $n$, as soon as this equality holds
for $n=1$ (where $\rk_p(A):=\dim_{\F_p}(A/A^p)$). For instance, let $k=\Q(\sqrt{-m})$ 
and $K \subset k (\mu_{257})$ and consider the stability 
of the $2$-ranks when $\rk_2({\mathcal C}) \geq 1$ ($\lambda(1)=0$ since 
$\order S \leq 2$ and $\rho = t \geq 1$ when ${\mathcal C} \ne 1$); 
we find many examples:

\smallskip
For $m \in \{$-15, -35, -95, -111, -123, -159, -215, -235, -259, -287, -291, 
-303, -327, -355, -371, -415, -447,$\,\ldots \}$ the $2$-rank is $1$ and stable from $n=0$.

\smallskip
For $m \in \{$-39, -87, -91,  -155, -183, -195, -203, -247, -267, -295, -339, -395, -399,
-403, -427, -435, -455, -471,$\,\ldots \}$ the $2$-rank (equal to $2$ or $3$) is stable 
from $n=1$.

\smallskip
For $m=-55$ (successive structures ${\sf [4], [4,4], [8,4,4,2], [16,8,8,4]}$),
$m=-115$ (structures ${\sf [2], [2,2], [2,2,2,2], [8,8,8,8]}$), we have
two examples of stabilization from $n=2$ (rank $4$).

\smallskip
We see that the orders are not necessarily $\lambda$-stable 
from $n=0$ ($\lambda = \order S - 1$):\par
\ft\begin{verbatim}
m=-15       m=-403           m=-259      m=-95       m=-195             m=-895     
v0=1 [2]    v0=1  [2]        v0=2 [4]    v0=3 [8]    v0=2  [2,2]        v0=4 [16]  
v1=2 [4]    v1=5  [8,2,2]    v1=3 [8]    v1=4 [16]   v1=4  [4,2,2]      v1=5 [32]  
v2=3 [8]    v2=8  [16,4,4]   v2=4 [16]   v2=5 [32]   v2=9  [8,8,8]      v2=6 [64]  
v3=4 [16]   v3=11 [32,8,8]   v3=5 [32]   v3=6 [64]   v3=12 [16,16,16]   v3=7 [128] 
\end{verbatim}\ns

If $k \subseteq K \subseteq \wt k$, for a $\Z_p$-extension $\wt k$ of $k$, and
if $\lambda(1) = 0$ (e.g., $k$ is any quadratic field and  ${\mathcal C} \ne 1$)
then $\rk_p({\mathcal C}_n) = \rk_p({\mathcal C})$ as soon as this is true for $n=1$
(this may holds from some layer). See \cite[Section 7]{Gra+} for many examples.
The following ones (among others) show a great tendency to stabilization of the rank:\par
\ft\begin{verbatim}
p=2 m=-93 [2,2],[2,2,2],[2,2,2,2,2],[2,2,2,2,2,2,2,2,2],[2,4,4,4,4,4,4,8,8]
p=3 m=6559 [9],[3,27],[9,27]; m=-362 [9],[3,27],[9,81]; m=-929 [9],[3,9,27],[9,27,81]
\end{verbatim}\ns

({\bf c}) If $K/\Q$ is abelian with $p \nmid [k : \Q]$, 
$\lambda$-stabilities may hold for the isotopic $\Z_p[\Gal(k/\Q)]$-components 
of the ${\mathcal X}_n$'s, using fixed points formulas \cite{Jau0, Jau1}. 

\smallskip
({\bf d}) Remarks. Let's consider, for an arbitrary $k$, the family 
$\{{\mathcal C}_n\}$ assuming the context of Lemma \ref{albert} and Remark 
\ref{torsion}; so $\lambda = \max\,(0, \order S - 1 - \rho)$ with $\rho =  r_1+r_2 - 1$.

\smallskip\noindent
\qquad {\bf (i)} The condition of $\lambda$-stability, $\order {\mathcal C}_1 = 
\order {\mathcal C} \cdot p^{\lambda}$, implies ${\mathcal C}_n^{G_n} = 
{\mathcal C}_n$ for all $n \geq 1$, but is not equivalent. Indeed, let $k = \Q(\sqrt m)$, 
$m>0$ such that $17$ splits in $k$, and let $K=kL$, where $L$ is the subfield of 
$\Q(\mu_{17})$ of degree $8$; whence $\rho = 1$, $\lambda = 0$. 
Assume that $\omega_1(\varepsilon) = 1$ ($\varepsilon$ norm in $k_1/k$);
then, from Chevalley's formula, $\order {\mathcal C}_1^{G_1} = \order {\mathcal C} \cdot p \ne 
\order {\mathcal C}$. For $m=26$, $\varepsilon = 3+\sqrt {26}
\equiv 5^2 \pmod {17}$ in the completion $\Q_{17}$; one obtains the structures 
$[2], [2,2], [4,2]$ giving $\order {\mathcal C}=2$ and $\order {\mathcal C}_1 
= \order {\mathcal C}_1^{G_1} = 4$.

\smallskip\noindent
\qquad {\bf (ii)}
When $\order {\mathcal C}_1 = \order {\mathcal C} \cdot p^{\lambda}$
does not hold (whence 
$\order {\mathcal C}_1 > \order {\mathcal C} \cdot p^{\lambda}$), 
we observe that orders and $p$-ranks are
unpredictable and possibly unlimited in the tower; indeed, put 
${\mathcal C}_n^i := \{c \in {\mathcal C}_n,\, c^{(1-\sigma_n)^i} = 1\} 
=: \cl_n({\mathcal I}_n^i)$, $i \geq 1$; we have, from \eqref{totram}, 
$\order \big({\mathcal C}_n^{i+1}/{\mathcal C}_n^i\big) = 
\frac{\order {\mathcal C}}{\order \Norm_n({\mathcal C}_n^i)}\! \times\!
\frac{p^{n\,(\oorder S -1)}}{\order \omega_n(\Lambda^i)}$, 
where $\Lambda^i := \{x \in k^\times, \,  (x) \in \Norm_n({\mathcal I}_n^i)\}$;
in general these two factors are non-trivial giving
$\order {\mathcal C}_n^{i+1} > \order {\mathcal C}_n^i$ and the filtration increases.

\section{Behavior of other invariants in cyclic $p$-towers}
Let $\Sigma_p$ be the set of $p$-places of $k$.
In relation with class field theory and cohomology of the Galois group
${\mathcal G}_{k,\Sigma_p}$ of the maximal $\Sigma_p$-ramified (= $p$-ramified)
pro-$p$-extension of $k$ and of its local analogues 
${\mathcal G}_{k_{\mathfrak p}}$, the fundamental 
Tate--Chafarevich groups \cite[Theorem 3.74]{K}, \cite[\S\,1]{Ng}: 
$$\hbox{$\Cha_k^i(\Sigma_p) := {\rm Ker} \big 
[{\Hom}^i ({\mathcal G}_{k,\Sigma_p},\F_p) \rightarrow \bigoplus_{{\mathfrak p}\in \Sigma_p} \, 
{\Hom}^i ({\mathcal G}_{k_{\mathfrak p}},\F_p) \big]$, $i = 1, 2$,}$$ 
can be examined, in the same way in towers, as for generalized class groups. 
This is possible for $\Cha_k^1(\Sigma_p) \simeq ({\mathcal C}_k/\cl_k(\Sigma_p)) \otimes \F_p$,
but $\Cha_k^2(\Sigma_p)$ does not appear immediately as a generalized class group
and is not of easy computational access (see \cite{HMR} for a thorough study 
of $\Cha_k^2(S)$ when $\Sigma_p \not\subset S$; the case $\Sigma_p \subseteq S$
easily reduces to $S=\Sigma_p$ under Leopoldt's conjecture 
\cite[Theorem III.4.1.5]{Gra3}). 
Nevertheless, we have an arithmetic
interpretation, much more easily computable and having well-known
properties, with the key invariant, closely related to $\Cha_k^2(\Sigma_p)$,
$${\mathcal T}_k := \tor_{\Z_p} ({\mathcal G}_{k,\Sigma_p}^\ab)  \simeq 
\Hom^2({\mathcal G}_{k,\Sigma_p},\Z_p)^\ast \ \,\hbox{\cite[Th\'eor\`eme 1.1]{Ng}}$$ 
(see general rank computations in \cite[Theorem III.4.2 and Corollaries]{Gra3}).
Nevertheless, ${\mathcal T}_k$, linked to the residue at $s=1$ of 
$p$-adic $\zeta$-functions \cite{Co}, gives rise to some analogies with generalized 
class groups (due to reflection theorems \cite[Theorem II.5.4.5]{Gra3}).
See in \cite{Gra5} the interpretation of ${\mathcal T}_k$
with the $p$-class group ${\mathcal C}_k$ and the normalized 
$p$-adic regulator ${\mathcal R}_k$.
Since the deep aspects of these cohomology groups are due to the $p$-adic 
properties of global units (i.e., ${\mathcal R}_k$) we shall restrict ourselves to 
the totally real case under Leopoldt's conjecture. For any field $\kappa$, let 
$\kappa_\infty$ be its cyclotomic $\Z_p$-extension; then ${\mathcal T}_\kappa = 
\Gal((H_\kappa^\pr)^\ab/\kappa_\infty)$ where $H_\kappa^\pr$ is the maximal 
abelian $\Sigma_p$-ramified pro-$p$-extension of $\kappa$.

\smallskip
Let $K/k$ be a totally real, $S$-ramified, cyclic $p$-tower; 
denote by $S^\ta \subseteq S$ the subset of tame places 
(when non-empty, {\it $S^\ta$ is assumed totally ramified}, 
but the subset of $p$-places may be arbitrary). 
For all extensions $L/F$ contained in $K/k$,
the transfer maps $\J_{L/F} : {\mathcal T}_F \to {\mathcal T}_L$ are injective 
\cite[Theorem IV.2.1]{Gra3} and,  since $L_\infty/F_\infty$ is totally $S^\ta$-ramified,
the norm maps $\Norm_{L/F} : {\mathcal T}_L \to {\mathcal T}_F$ are surjective.
In the particular case $L/F \subset k_\infty/k$, $S^\ta = \es$, but
$\Norm_{L/F}$ is still surjective (indeed, $F_\infty = L_\infty = k_\infty$ and 
$(H_k^\pr)^\ab \subseteq (H_F^\pr)^\ab \subseteq (H_L^\pr)^\ab$).

\begin{lemma}
Let $K/k$ be a totally real cyclic $p$-tower of degree $q \,p^N$, $p$-ramified 
or else totally ramified at $S^\ta \ne \es$, where $q=p$ or $4$; assume that
$k \cap \Q_\infty = \Q$. Then, for any $0 \leq n \leq n + m \leq N$, 
and $g_n^m := \Gal(k_{n+m}/k_n)$, $\order {\mathcal T}_{n+m}^{g_n^m} = 
\order {\mathcal T}_n \cdot p^{m\cdot \oorder S^\ta}$.
\end{lemma}

\begin{proof}
We have (with our notations) $\order {\mathcal T}_{n+m}^{g_n^m} = 
\order {\mathcal T}_n \cdot p^{m\cdot \oorder S^\ta} \!\!\cdot p^{R - m}$ 
\cite[Theorem IV.3.3, Exercise 3.3.1]{Gra3}, 
$R = \min_{{\mathfrak l} \in S^\ta}(m \,; \ldots, 
\nu_{\mathfrak l}+\varphi_{\mathfrak l}+\gamma_{\mathfrak l} , \ldots)$ 
(with $p^{\nu_{\mathfrak l}} \sim q^{-1}\log(\ell)$ where ${\mathfrak l} \cap \Z 
= \ell \,\Z$, $p^{\varphi_{\mathfrak l}} \sim$ residue degree of ${\mathfrak l}$ 
in $k/\Q$, $\gamma_{\mathfrak l} =0$ since $S^\ta$ does not split in $K/k$). 
When $S^\ta \ne \es$, {\it the existence} (see \cite[V, Theorem 2.9]{Gra3} for a 
characterization) of a cyclic $p$-tower $K/k$ implies 
${\mathcal N} ({\mathfrak l}) \equiv 1 \pmod {q\,p^N}$, 
for all ${\mathfrak l} \in S^\ta$, where 
${\mathcal N}$ is the absolute norm; indeed, the inertia group of the local extension 
$K_{\mathfrak L}/k_{\mathfrak l}$, ${\mathfrak L} \mid {\mathfrak l}$, is
isomorphic to a subgroup of $\F_{\!\!{\mathcal N} ({\mathfrak l})}^\times$.
So, we have $\nu_{\mathfrak l}+\varphi_{\mathfrak l} \geq N$ and $R=m$; this 
explains the limitation of the level $n$ to get the given fixed points formula in any case.
If $k \cap \Q_\infty = \Q_{n_0}$, one must take $\ell  \equiv 1 \pmod {q p^{n_0} \cdot p^N}$
\end{proof}

If a $\lambda$-stability does exist with ${\mathcal T}_n = {\mathcal T}_n^{G_n}$
of order $\order {\mathcal T} \cdot p^{\lambda \cdot n}$, then $\lambda = \order S^\ta$
(from the relation $\order {\mathcal T}_1 = \order {\mathcal T}\,p^{\oorder S^\ta}$ 
 giving ${\mathcal T}_n^{G_n} = {\mathcal T}_n^{G_n^p}$).
However, the properties of the groups ${\mathcal T}_n$ will give sometimes
non-trivial Iwasawa's $\wt \mu$-invariants in $k_\infty/k$ and very large ranks 
in cyclic $p$-towers with tame ramification; moreover $\wt \lambda$ may be much 
larger than $\lambda=\order S^\ta$ which explains that the $\lambda$-stability 
problem is less pertinent in this framework (few examples). A reason 
(based on the properties,  in $k_{n+m}/k_n$, of the norms 
$\Norm_n^m$ (surjectivity) and transfer maps $\J_n^m$ (injectivity)) is 
the following, where $g_n^m := \Gal(k_{n+m}/k_n)$:

\smallskip
From $1 \to  \J_n^m{\mathcal T}_n \to
{\mathcal T}_{n+m} \to {\mathcal T}_{n+m}/ \J_n^m {\mathcal T}_n \to 1$, we get:

\medskip
\centerline{$1  \to  {\mathcal T}_{n+m}^{g_n^m} / \J_n^m {\mathcal T}_n  \to  
({\mathcal T}_{n+m}/ \J_n^m {\mathcal T}_n)^{g_n^m}  \to  {\rm H}^1(g_n^m, \J_n^m {\mathcal T}_n)
\to  {\rm H}^1(g_n^m, {\mathcal T}_{n+m})$, }

\medskip\noindent
where ${\rm H}^1(g_n^m, \J_n^m {\mathcal T}_n) = (\J_n^m {\mathcal T}_n) [p^m]
\simeq {\mathcal T}_n[p^m]$, and:

\medskip
\centerline{$\order {\rm H}^1(g_n^m, {\mathcal T}_{n+m}) 
= \order {\rm H}^2(g_n^m, {\mathcal T}_{n+m})$
$= \order ({\mathcal T}_{n+m}^{g_n^m} / \J_n^m {\mathcal T}_n) = p^{m \cdot \oorder S^\ta}$,}

\medskip\noindent
giving an exact sequence of the form
$1\to A \to ({\mathcal T}_{n+m}/ \J_n^m {\mathcal T}_n)^{g_n^m}  
\to {\mathcal T}_n[p^m] \to A'$, with $\order A' = \order A = p^{m\cdot\oorder S^\ta}$. 
We then obtain the following inequality:
\begin{equation}\label{mu}
\order {\mathcal T}_{n+m}\geq \order {\mathcal T}_n \cdot \order {\mathcal T}_n[p^m] ,
\end{equation}
more precise than $\order {\mathcal T}_{n+m} \geq \order {\mathcal T}_{n+m}^{g_n^m} =
 \order {\mathcal T}_n \cdot p^{m \cdot \order S^\ta}$, when the $p$-rank of ${\mathcal T}_n$
increases, since $\order S^\ta$ is a constant.

\smallskip
For $m=1$ one gets $\order {\mathcal T}_{n+1}\geq 
\order {\mathcal T}_n \cdot p^{\rk_p({\mathcal T}_n)}$.
This tendency to give large orders and $p$-ranks is enforced by the presence of tame
ramification (see \S\,\ref{bigrank} for numerical illustrations) and leads to the vast 
theory of pro-$p$-extensions and Galois cohomology of pro-$p$-groups in a broader 
context than $p$-ramification theory, which is not our purpose (see \cite{HM} 
for more information). This explains that, for $\lambda = \order S^\ta = 0$, 
Corollary \ref{capitulation} never applies for ${\mathcal T}$ 
(except ${\mathcal T}=1$ giving ${\mathcal T}_n=1$ for all~$n$).

\smallskip
It would be interesting to give an analogous study for the Jaulent
logarithmic class group \cite{Jau3} since its capitulation in the cyclotomic 
$\Z_p$-extension of $k$ (real) characterizes Greenberg's conjecture.

\begin{remark}
For $p$-class groups ${\mathcal C}_n$, in totally real cyclic $p$-towers, 
one obtains an inequality similar to \eqref{mu} provided one replaces 
${\mathcal C}_n$ by $\J_n^m{\mathcal C}_n$, where $\J_n^m$ is 
{\it in general non-injective}, a crucial point in this comparison. 
But for CM-fields $k$ and $p \ne 2$, {\it $\J_n$ is injective on the ``minus parts''},
giving $\order {\mathcal C}_{n+m}^-\geq \order {\mathcal C}_n^- 
\cdot \order {\mathcal C}_n^-[p^m]$, which explains the results for 
imaginary quadratic fields in \S\,\ref{quad}, especially for $m=-782$ and 
$m=-858$, for which we can predict $\order{\mathcal C}_4$ larger than 
$2^{54}$ and $2^{46}$, respectively. 
\end{remark}

\begin{remark} \label{overQ}
One can wonder about the contribution of the $p$-towers 
contained in $\Q(\mu_\ell)$, $\ell \equiv 1 \pmod {q p^N}$,
used in the compositum with $k$. We have the following 
examples of ${\mathcal T}_n$ for $p=2,3$ and $S = S_\ell$;
let ${\sf vn, rkn}$ denote the $p$-valuation, the $p$-rank, of 
${\mathcal T}_n$, respectively; the parameter $E$ must be chosen such
that $E > e+1$, where $p^e$ is the exponent of ${\mathcal T}_n$.\par
\ft\begin{verbatim}
{p=3;ell=17497;E=8;print("p=",p," ell=",ell);for(n=0,4,Qn=polsubcyclo(ell,p^n);
kn=bnfinit(Qn,1);knmod=bnrinit(kn,p^E);Tn=knmod.cyc;\\computation of Tn
T=List;d=matsize(Tn)[2];rkn=0;vn=0;for(j=1,d-1,c=Tn[d-j+1];w=valuation(c,p);
if(w>0,vn=vn+w;rkn=rkn+1;listinsert(T,p^w,1)));\\computation of vn and rkn
print("v",n,"=",vn," rk",n,"=",rkn," ",T))}
\end{verbatim}\ns
\ft\begin{verbatim}
p=2,ell=17        p=2,ell=7681      p=2,ell=257 
v0=0 rk0=0 []     v0=0 rk0=0 []     v0=0  rk0=0  []  
v1=1 rk1=1 [2]    v1=1 rk1=1 [2]    v1=3  rk1=1  [8]   
v2=2 rk2=1 [4]    v2=2 rk2=1 [4]    v2=7  rk2=3  [16,4,2] 
v3=3 rk3=1 [8]    v3=3 rk3=1 [8]    v3=13 rk3=7  [32,8,2,2,2,2,2]  
v4=4 rk4=1 [16]   v4=4 rk4=1 [16]   v4=23 rk4=15 [64,16,2,2,2,2,2,2,2,2,2,2,2,2,2]  
                                    v5=38 rk5=15 [128,32,4,4,4,4,4,4,4,4,4,4,4,4,4] 
\end{verbatim}\ns
\ft\begin{verbatim}
p=3,ell=109       p=3,ell=163       p=3,ell=487         p=3,ell=17497     
v0=0 rk0=0 []     v0=0 rk0=0 []     v0=0 rk0=0 []       v0=0 rk0=0 [] 
v1=1 rk1=1 [3]    v1=1 rk1=1 [3]    v1=2 rk1=2 [3,3]    v1=3 rk1=2 [9,3] 
v2=2 rk2=1 [9]    v2=2 rk2=1 [9]    v2=4 rk2=2 [9,9]    v2=6 rk2=3 [27,9,3] 
v3=3 rk3=1 [27]   v3=3 rk3=1 [27]    
\end{verbatim}\ns
\end{remark}

\subsection{Case of the cyclotomic $\Z_p$-extensions $k_\infty/k$ 
of $k$ real}\label{rankSp}
The Iwasawa invariants $\wt \lambda$, $\wt \mu$, $\wt \nu$, attached to 
$\ds \limproj {\mathcal T}_n$ satisfy $\wt \lambda = \wt \mu = 0$
if and only if ${\mathcal T} = 1$ (in which case $\wt \nu = 0$).
Indeed, if $\wt \lambda = \wt \mu=0$, $\order {\mathcal T}_n= p^{\wt \nu}$ 
for all $n \gg 0$; but \eqref{mu}, with $m=1$, gives $\wt \nu \geq \wt \nu + r_n$, 
where $r_n := \rk_p({\mathcal T}_n)$, whence $r_n=0$; but $r_n \geq r$,
so $r=0$ ($\J_n$ is injective). The reciprocal comes from
${\mathcal T}_n^{G_n} \simeq {\mathcal T}$ since $S^\ta = \es$. 
Thus a $\lambda$-stability does exist only for $\lambda=0$, which is 
excluded if ${\mathcal T} \ne 1$.
Computations suggest that the true Iwasawa parameters $\wt \lambda$, 
$\wt \mu$ may be often non-zero, even if few levels are accessible with 
PARI/GP; often there is stability for the $p$-class groups ${\sf Cn}$
($\order{\sf Sp}$ is the number of $p$-places of $k=\Q(\sqrt m)$):\par
\ft\begin{verbatim}
{p=3;E=8;N=4;bm=2;Bm=10^3;for(m=bm,Bm,if(core(m)!=m,next);P=x^2-m;
Sp=1;if(kronecker(m,p)==1,Sp=2);print();print("p=",p," m=",m," #Sp=",Sp);
for(n=0,N,if(p==2,Qn=x;for(k=1,n,Qn=Qn^2-2));if(p!=2,Qn=polsubcyclo(p^(n+1),p^n));
Pn=polcompositum(Qn,P)[1];kn=bnfinit(Pn,1);Cn=kn.cyc;C=List;d=matsize(Cn)[2];
for(j=1,d,c=Cn[j];w=valuation(c,p);if(w>0,listinsert(C,p^w,1)));\\p-class group Cn
knmod=bnrinit(kn,p^E);rkn=0;vn=0;Tn=knmod.cyc;\\computation of Tn
T=List;d=matsize(Tn)[2];for(j=1,d-1,c=Tn[d-j+1];w=valuation(c,p);
if(w>0,vn=vn+w;rkn=rkn+1;listinsert(T,p^w,1)));\\computation of vn and rkn
print("v",n,"=",vn," rk",n,"=",rkn," T",n,"=",T," C",n,"=",C)))}
\end{verbatim}\ns
\ft\begin{verbatim}
p=2
m=33,#Sp=2            m=41,#Sp=2             m=217,#Sp=2
v0=1 rk0=1 [2]  C0=[] v0=4 rk0=1 [16] C0=[]  v0=2  rk0=2 [2,2]                 C0=[]
v1=4 rk1=1 [16] C1=[] v1=5 rk1=1 [32] C1=[2] v1=4  rk1=3 [4,2,2]               C1=[2]
v2=5 rk2=1 [32] C2=[] v2=6 rk2=1 [64] C2=[4] v2=8  rk2=5 [4,4,4,2,2]           C2=[2]
v3=6 rk3=1 [64] C3=[] v3=7 rk3=1 [128]C3=[8] v3=15 rk3=9 [8,8,8,2,2,2,2,2,2]   C3=[2]
v4=7 rk4=1 [128]C4=[] v4=8 rk4=1 [256]C4=[8] v4=24 rk4=9 [16,16,16,4,4,4,4,4,4]C4=[2]
\end{verbatim}\ns
\ft\begin{verbatim}
m=193,#Sp=2                                           m=69,#Sp=1
v0=1  rk0=1  [2] C0=[]                                v0=2  rk0=1 [4]     C0=[]
v1=2  rk1=2  [2,2] C1=[]                              v1=4  rk1=2 [4,4]   C1=[]
v2=4  rk2=4  [2,2,2,2] C2=[]                          v2=6  rk2=2 [8,8]   C2=[]
v3=8  rk3=8  [2,2,2,2,2,2,2,2] C3=[]                  v3=8  rk3=2 [16,16] C3=[]
v4=16 rk4=15 [4,2,2,2,2,2,2,2,2,2,2,2,2,2,2] C4=[] 
\end{verbatim}\ns
\ft\begin{verbatim}
p=3
m=62,#Sp=1              m=103,#Sp=2                  m=1714,#Sp=2
v0=1 rk0=[3]    C0=[]   v0=1 rk0=1 [3]    C0=[]      v0=2 rk0=2  [3,3]      C0=[3]
v1=3 rk1=[9,3]  C1=[]   v1=3 rk1=2 [9,3]  C1=[3,3]   v1=5 rk1=4  [9,3,3,3]  C1=[3,3]
v2=5 rk2=[27,9] C2=[]   v2=5 rk2=2 [27,9] C2=[3,3]   v2=9 rk2=4  [27,9,9,9] C2=[9,3]
\end{verbatim}\ns
Formula \eqref{mu} predicts $\order {\mathcal T}_5 \geq  
2^{16+15} = 2^{31}$ for $m=193$, giving possibly $\wt \mu = 1$.

\subsection{Case of real cyclic $p$-towers with tame ramification}\label{bigrank}
In the case of $S_\ell$-ramified cyclic $p$-towers, one gets few examples 
of $\lambda$-stability (for $\lambda = \order S_\ell$) and mainly spectacular groups 
${\mathcal T}_n$; for example let $k=\Q(\sqrt m)$ and $K = k (\mu_{257})$ with $p=2$, 
then $K = k (\mu_{109})$ and $K = k (\mu_{163})$ with $p=3$:

\smallskip
(i) Case $p=2$.\par
\ft\begin{verbatim}
{p=2;ell=257;N=3;E=8;bm=2;Bm=10^3;for(m=bm,Bm,if(core(m)!=m||Mod(m,ell)==0,next);
P=x^2-m;S=1;if(kronecker(m,ell)==1,S=2);print("p=",p," ell=",ell," m=",m," #S=",S);
for(n=0,N,Pn=polcompositum(polsubcyclo(ell,p^n),P)[1];kn=bnfinit(Pn,1);\\layer kn
Cn=kn.cyc;C=List;d=matsize(Cn)[2];for(j=1,d,c=Cn[j];w=valuation(c,p);
if(w>0,listinsert(C,p^w,1)));\\p-class group Cn
knmod=bnrinit(kn,p^E);Tn=knmod.cyc;T=List;d=matsize(Tn)[2];\\computation of Tn
rkn=0;vn=0;for(j=1,d-1,c=Tn[d-j+1];w=valuation(c,p);if(w>0,rkn=rkn+1;vn=vn+w;
listinsert(T,p^w,1)));print("v",n,"=",vn," rk",n,"=",rkn," ",T," C",n,"=",C)))}
\end{verbatim}\ns
\ft\begin{verbatim}
p=2,ell=257,m=17,#S=2
v0=1  rk0=1  [2] C0=[]
v1=7  rk1=3  [16,4,2] C1=[]
v2=16 rk2=7  [32,8,8,4,2,2,2] C2=[]
v3=31 rk3=15 [64,16,16,4,4,4,4,4,2,2,2,2,2,2,2] C3=[]
v4=58 rk4=31 [128,32,32,8,8,8,8,4,4,4,4,4,2,2,2,2,2,2,2,2,2,2,2,2,2,2,2,2,2,2,2] C4=[]
\end{verbatim}\ns
\ft\begin{verbatim}
p=2,ell=257,m=34,#S=2
v0=1  rk0=1  [2] C0=[2]
v1=8  rk1=3  [16,8,2] C1=[4,2]
v2=20 rk2=7  [32,16,8,8,4,4,2] C2=[4,2,2,2]
v3=37 rk3=15 [64,32,16,16,8,4,4,4,4,4,2,2,2,2,2] C3=[4,2,2,2,2,2,2,2]
v4=65 rk4=27 [128,64,32,32,16,8,8,8,8,8,4,4,4,4,4,4,2,2,2,2,2,2,2,2,2,2,2] 
                                                             C4=[8,4,4,2,2,2,2,2,2,2]
\end{verbatim}\ns

Strikingly, for $m=17$, one verifies the formula
$\order {\mathcal T}_n = 2^{3 \cdot n + 3 \cdot 2^n -2}$ for $n \in [0, 4]$ and
\eqref{mu} predicts $\order {\mathcal T}_5 \geq 2^{89}$; for $m=34$,
\eqref{mu} predicts $\order {\mathcal T}_5 \geq 2^{92}$ (or much more).

\smallskip
(ii) Case $p=3$.\par
\ft\begin{verbatim}
p=3,ell=163,m=2,#S=1     p=3,ell=163,m=10,#S=2       p=3,ell=109,m=29,#S=2
v0=0 rk0=1 []  C0=[]     v0=0 rk0=0 []    C0=[]      v0=2 rk0=1 [9]    C0=[]
v1=1 rk1=1 [3] C1=[]     v1=2 rk1=2 [3,3] C1=[]      v1=4 rk1=2 [27,3] C1=[]
v2=2 rk2=1 [9] C2=[]     v2=4 rk2=2 [9,9] C2=[]      v2=6 rk2=2 [81,9] C2=[]
\end{verbatim}\ns
\ft\begin{verbatim}
p=3,ell=109,m=15,#S=2               p=3,ell=109,m=145,#S=2
v0=1  rk0=1 [3]         C0=[]       v0=0  rk0=0 []                 C0=[]
v1=6  rk1=4 [9,9,3,3]   C1=[]       v1=8  rk1=3 [81,27,3]          C1=[]
v2=10 rk2=4 [27,27,9,9] C2=[]       v2=15 rk2=7 [243,81,9,3,3,3,3] C2=[]
\end{verbatim}\ns
\ft\begin{verbatim}
p=3,ell=163,m=15,#S=2       p=3,ell=163,m=79 #S=1        p=3,ell=163,m=118,#S=2
v0=1 rk0=1 [3]       C0=[]  v0=2 rk0=1 [9]        C0=[3] v0=0 rk0=0 []         C0=[]
v1=5 rk1=4 [9,3,3,3] C1=[3] v1=7 rk1=4 [27,9,3,3] C1=[9] v1=4 rk1=3 [9,3,3]    C1=[]
v2=9 rk2=4 [27,9,9,9]C2=[9] v2=11rk2=4 [81,27,9,9]C2=[9] v2=8 rk2=4 [27,9,9,3] C2=[]
\end{verbatim}\ns

See Remark \ref{overQ} to compute the contributions of the towers over $\Q$.

\begin{remark} Let $K/k$ be a totally real, $p$-ramified (whence $S \subseteq 
\Sigma_p$, arbitrary), cyclic $p$-tower (i.e., $k \subseteq K \subseteq (H_k^\pr)^\ab$, 
where $\Gal((H_k^\pr)^\ab/k) \simeq \Z_p \oplus {\mathcal T}_k$). Then if the 
$p$-ranks of $k$ and $k_1$ coincide, the $p$-rank is constant in the tower 
(indeed, this corresponds to the stability of ${\mathcal T}_n[p]$, $n \in [0, N]$, since
${\mathcal T}_n[p]^{G_n} = {\mathcal T}[p]$ and ${\mathcal T}_n[p]^{G_n^p} 
= {\mathcal T}_1[p]$ which leads to ${\mathcal T}_n[p] \simeq {\mathcal T}[p]$
for all $n$). Of course, this may occur only from some layer $k'$ and the
numerical data of \S\,\ref{rankSp} gives many examples.
\end{remark}



\begin{thebibliography}{xx}

\bibitem{Bcap} S. Bosca, \textit{Principalization of ideals in Abelian extensions 
of number fields}, Int. J. Number Th. {\bf 5} (2009), 527--539. \hspace{0.1cm}
\url{https://doi.org/10.1142/S1793042109002213}

\bibitem{Che} C. Chevalley, \textit{Sur la th\'eorie du corps de classes dans les corps finis 
et les corps locaux}, Jour. of the Faculty of Sciences Tokyo {\bf 2} (1933), 365--476.
\url{http://eudml.org/doc/192833}\hspace{0.1cm}

\bibitem{Co} J. Coates, \textit{$p$-adic $L$-functions and Iwasawa's theory}, Algebraic 
number fields: $L$-functions and Galois properties (Proc. Sympos., Univ. Durham, 
Durham, 1975), pp. 269--353. Academic Press, London, 1977.
\url{https://lib.ugent.be/catalog/rug01:000005363}\hspace{0.1cm}

\bibitem{FS} E. Friedman, J.W. Sands, \textit{On the
$\ell$-adic Iwasawa $\lambda$-invariant in a $p$-extension, with an Appendix
by L.C. Washington}, Math. Comp. {\bf 64}(212) (1995), 1659--1674.
\url{https://doi.org/10.2307/2156}\hspace{0.1cm}

\bibitem{Fu} T. Fukuda, \textit{Remarks on $\Z_p$-extensions of number fields}, 
Proc. Japan Acad. Ser. A {\bf 70}(8) (1994), 264--266.
\url{https://doi.org/10.3792/pjaa.70.264}\hspace{0.1cm}

\bibitem{Gra0} G. Gras, \textit{Classes g\'en\'eralis\'ees invariantes},  J. Math. Soc. 
Japan {\bf 46}(3) (1994), 467--476.
\url{http://doi.org/10.2969/jmsj/04630467}\hspace{0.1cm}

\bibitem{Gcap} G. Gras, \textit{Principalisation d'id\'eaux par extensions 
absolument ab\'eliennes}, J. Number Th. {\bf 62} (1997), 403--421. \hspace{0.1cm}
\url{https://doi.org/10.1006/jnth.1997.2068}

\bibitem{Gra3} G. Gras, \textit{Class Field Theory: from theory to practice}, 
corr. 2nd ed. Springer Monographs in Mathematics, Springer, 
xiii+507 pages (2005).\hspace{0.1cm}

\bibitem{Gra1} G. Gras, \textit{Invariant generalized ideal classes -- Structure theorems for
$p$-class groups in $p$-extensions}, Proc. Math. Sci. {\bf 127}(1)  (2017), 1--34.
\url{http://doi.org/10.1007/s12044-016-0324-1}\hspace{0.1cm}

\bibitem{Gra5} G. Gras, \textit{The $p$-adic Kummer--Leopoldt Constant: Normalized 
$p$-adic Regulator}, Int. J. of Number Theory {\bf 14}(2) (2018), 329--.
\url{https://doi.org/10.1142/S1793042118500203}\hspace{0.1cm}

\bibitem{Gra2} G. Gras, \textit{Algorithmic complexity of Greenberg's conjecture},
Arch. Math. (2021). \url{https://doi.org/10.1007/s00013-021-01618-9}\hspace{0.1cm}

\bibitem{Gra+} G. Gras, \textit{Numerical data verifying capitulations of ${\mathcal C}$}, 
Dropbox document (2021).
\url{https://www.dropbox.com/s/ypnc3892j80n8jp/Data.37.pdf?dl=0}\hspace{0.1cm}

\bibitem{Gre}  R. Greenberg, \textit{On the Iwasawa invariants of totally real 
number fields}, Amer. J. Math. {\bf 98}(1) (1976),  263--284.
\url{https://doi.org/10.2307/2373625}\hspace{0.1cm}

\bibitem{HM} F. Hajir, C. Maire,
\textit{On the invariant factors of class groups in towers of number fields},
Canad. J. Math. {\bf 70}(1) (2018), 142--172.
\url{https://doi.org/10.4153/CJM-2017-032-9}\hspace{0.1cm}

\bibitem{HMR} F. Hajir, C. Maire, R.  Ramakrishna, \textit{On the Shafarevich group 
of restricted ramification extensions of number fields}, Indiana Univ. Math. J. 
(2021) (to appear).\hspace{0.1cm}

\bibitem{Jau0} J.-F. Jaulent, \textit{L'arithm\'etique des $\ell$-extensions}, 
Pub. Math. Fac. Sci. Besan\c con, Th\'eorie des Nombres 1985/86 (1986). 
\url{https://doi.org/10.5802/pmb.a-42}\hspace{0.1cm}

\bibitem{Jau1} J.-F. Jaulent, \textit{Th\'eorie $\ell$-adique globale du  corps de classes}, 
J. Th\'eorie Nombres Bordeaux {\bf 10}(2) (1998), 355--397.
\url{https://doi.org/10.5802/jtnb.233}\hspace{0.1cm}

\bibitem{Jau2} J.-F. Jaulent, \textit{Abelian capitulation of ray class groups} (2020).
\url{https://arxiv.org/abs/1801.07173}\hspace{0.1cm}

\bibitem{Jau3} J-F. Jaulent, \textit{Classes logarithmiques des corps 
de nombres}, J. Th\'eor. Nombres Bordeaux {\bf 6}(2) (1994), 301--325. 
\url{https://doi.org/10.5802/jtnb.117 }\hspace{0.1cm}

\bibitem{Jcap} J-F. Jaulent, \textit{Principalisation ab\'elienne des groupes de 
classes logarithmiques}, Functiones et Approximatio {\bf 61} (2019), 257--275.
\url{https://www.math.u-bordeaux.fr/~jjaulent/Articles/PCLog.pdf} \hspace{0.1cm}

\bibitem{K} H. Koch, \textit{Galois theory of $p$-extensions}
(English translation of \emph{``Galoissche Theorie der $p$-Erweiterungen''}, 1970),
Springer Monographs in Math., Springer, 2002.\hspace{0.1cm}

\bibitem{KP} P. Koymans, C. Pagano, \textit{On the distribution of $Cl(K)[\ell^\infty]$
for degree $\ell$ cyclic fields}, J. Eur. Math. Soc. (to appear).
\url{https://arxiv.org/pdf/1812.06884}\hspace{0.1cm}

\bibitem{KS}  J.S. Kraft, R. Schoof,  \textit{Computing Iwasawa modules of real quadratic 
number fields}, Compositio Math. {\bf 97}(1-2) (1995), 135--155. 
\url{http://eudml.org/doc/90370}\hspace{0.1cm}

\bibitem{Kcap} M. Kurihara, \textit{On the ideal class group of the maximal real 
subfields of number fields with all roots of unity}, J. European Math. Soc. {\bf 1}(1)
(1999), 35--49. \hspace{0.1cm}
\url{https://doi.org/10.1007/PL00011159}

\bibitem{LOXZ} J. Li, Y. Ouyang, Y. Xu, S. Zhang, \textit{$\ell$-Class groups of fields in 
Kummer towers} (2020).
\url{https://arxiv.org/pdf/1905.04966v2.pdf}\hspace{0.1cm}

\bibitem{LY} J. Li, C.F. Yu, \textit{The Chevalley--Gras formula over global fields},
J. Th\'eor. Nombres Bordeaux {\bf 32}(2) (2020), 525--543.
\url{https://doi.org/10.5802/jtnb.1133}\hspace{0.1cm}

\bibitem{Miz} Y. Mizusawa, \textit{Tame pro-2 Galois groups and the basic $\Z_2$-extension}, 
Trans. Amer. Math. Soc. {\bf 370}(4) (2018),  2423--2461.
\url{https://doi.org/10.1090/tran/7023}\hspace{0.1cm}

\bibitem{MY} Y. Mizusawa, K. Yamamoto, \textit{On $p$-class groups of relative cyclic 
$p$-extensions}, Arch. Math. {\bf 117}(3) (2021), 253--260.
\url{https://doi.org/10.1007/s00013-021-01619-8}\hspace{0.1cm}

\bibitem{Ng} T. Nguyen Quang Do, \textit{Sur la $\Z_p$-torsion de certains 
modules galoisiens}, Ann. Inst. Fourier {\bf 36}(2) (1986), 27--46. 
\url{https://doi.org/10.5802/aif.1045}\hspace{0.1cm}

\bibitem{PARI} The PARI Group, \textit{PARI/GP, version 2.9.0)}, 
Universit\'e de Bordeaux (2016). 
\url{http://pari.math.u-bordeaux.fr/}\hspace{0.1cm}

\bibitem{Sm} A. Smith, \textit{$2^\infty$-Selmer groups, $2^\infty$-class groups
and Goldfeld's conjecture} (2017). 
\url{https://arxiv.org/abs/1702.02325}\hspace{0.1cm}

\bibitem{Ta} H. Taya, \textit{Computation of $\Z_3$-invariants of real quadratic fields},
Math. Comp. {\bf 65}(214), 779--784.
\url{https://doi.org/10.1090/S0025-5718-96-00721-1}\hspace{0.1cm}

\bibitem{Wa} L.C. Washington, \textit{The non-$p$-part of the class 
number in a cyclotomic $\Z_p$-extension}, Invent. Math. {\bf 49}(1) (1978), 87--97. 
\url{https://doi.org/10.1007/BF01399512}\hspace{0.1cm}

\end{thebibliography}
\end{document}